\documentclass[11pt,a4paper,reqno]{amsart}
\usepackage{amsfonts}\usepackage{amsmath,amssymb,amsthm,amsxtra}
\usepackage{float}
\usepackage{amssymb}
\usepackage{mathabx}
\usepackage{bbm}
\usepackage[colorlinks, linkcolor=blue,anchorcolor=Periwinkle,
citecolor=Red,urlcolor=Emerald]{hyperref}
\usepackage[usenames,dvipsnames]{xcolor}
\usepackage{enumitem}
\usepackage{geometry,array} 
\usepackage{graphicx}
\usepackage{subfigure}
\usepackage{bookmark}
\usepackage{tikz}\usetikzlibrary{matrix}
\usepackage{url}
\usepackage{dsfont}
\usepackage[colorinlistoftodos]{todonotes}
\usepackage{tablists} \restorelistitem
\usepackage{color}

\usetikzlibrary{decorations.markings}
\tikzset{->-/.style={decoration={  markings,  mark=at position #1 with
    {\arrow{>}}},postaction={decorate}}}
\tikzset{-<-/.style={decoration={  markings,  mark=at position #1 with
    {\arrow{<}}},postaction={decorate}}}
\usepackage{extarrows}
\usepackage[all]{xy}
\usepackage{setspace}
\usepackage{thmtools}
\usepackage{thm-restate}
\usepackage{hyperref}
\usepackage{cleveref}
\usepackage{pifont}

\newcommand{\mfd}{\mathbf{d}}
\newcommand{\mfe}{\mathbf{e}}
\newcommand{\mff}{\mathbf{f}}

\newcommand{\mfp}{\mathbf{p}}

\newcommand{\mfx}{\mathbf{x}}
\newcommand{\mfy}{\mathbf{y}}

\newcommand{\mcA}{\mathcal{A}}

\newcommand{\mcF}{\mathcal{F}}

\newcommand{\mcP}{\mathcal{P}}

\newcommand{\mbN}{\mathbb{N}}

\newcommand{\mbP}{\mathbb{P}}
\newcommand{\mbQ}{\mathbb{Q}}
\newcommand{\mbR}{\mathbb{R}}

\newcommand{\mbT}{\mathbb{T}}

\newcommand{\mbZ}{\mathbb{Z}}
\theoremstyle{plain}
\newtheorem{theorem}{Theorem}[section]

\newtheorem{lemma}[theorem]{Lemma}

\newtheorem{proposition}[theorem]{Proposition}
\newtheorem{conjecture}[theorem]{Conjecture}
\theoremstyle{definition}
\newtheorem{definition}[theorem]{Definition}

\newtheorem{example}[theorem]{Example}

\newtheorem{remark}[theorem]{Remark}

\numberwithin{equation}{section}
\newtheorem{definition-proposition}[theorem]{Definition-Proposition}

\begin{document}

\title{Log-concavity of cluster algebras of type $A_n$}

\date{\today}
\author{Zhichao Chen}
\address{School of Mathematical Sciences\\ University of Science and Technology of China \\ Hefei, Anhui 230026, P. R. China}
\email{czc98@mail.ustc.edu.cn}

\author{Guanhua Huang}
\address{School of Mathematical Sciences\\ University of Science and Technology of China \\ Hefei, Anhui 230026, P. R. China}
\email{huanggh@mail.ustc.edu.cn}

\author{Zhe Sun}
\address{School of Mathematical Sciences\\ University of Science and Technology of China \\ Hefei, Anhui 230026, P. R. China}
\email{sunz@ustc.edu.cn}

\maketitle

\begin{abstract}
Okounkov \cite{Oko03} conjectured the log-concavity about the structure constants for many interesting basis from representation theory. For the cluster algebra, Gross, Hacking, Keel, Kontsevich \cite{GHKK18} introduced the atomic theta basis. We prove that the coefficients of the exponents of any cluster variable of type $A_n$ are log-concave. 
We show that the structure constants for theta basis of type $A_2$ are log-concave. As for larger generality, we conjecture that the log-concavity of the structure constants for theta basis of the cluster algebra.
\end{abstract}
%=======================================
\tableofcontents
%=======================================
\section{Introduction}
In \cite{Oko03}, Okounkov conjectured the log-concavity in large from statistical physics point of view, particularly, the structure constants for many interesting basis from representation theory. The recently developed Lorentzian polynomials \cite{BH20} help to solve many interesting cases, e.g. \cite{HMMS22}. We are interested in the log-concavity for the theta basis \cite{GHKK18} of cluster algebras, which are not Lorentzian in general. 

Cluster algebras are important commutative algebras with different generators and relations between them. In \cite{FZ02, FZ03}, they were first introduced to investigate the total positivity of Lie groups and canonical bases of quantum groups. Nowadays, cluster algebras are closely related to different subjects in mathematics.

Later on, in \cite{FG06, FG09}, Fock and Goncharov generalized the cluster structure into the cluster ensemble structure for a pair of dual spaces $(\mathcal{X}^*, \mathcal{A})$, and they conjectured that the tropical integer points of one space cluster modular group equivariantly parameterizes the canonical linear basis of the ring of regular functions on the dual space, and the highest term exponent of the regular function in cluster variables reflects the corresponding tropical point. Unfortunately, this conjecture is usually false due to lack of global functions by \cite{GHK15}. By the seminal work of Gross, Hacking, Keel and Kontsevich \cite{GHKK18}, using scattering diagrams, broken lines and theta functions, they proved the duality conjecture under certain conditions. By \cite{Man17}, the theta functions are the atomic global Laurent polynomials of cluster variables.
%Recently, Mandel and Qin \cite{MQ23} showed that for $\operatorname{PGL}_2$, the theta bases and the Fock--Goncharov's trace function bracelet bases are the same.
%The highest term exponent of the theta function in cluster variables reflects the corresponding tropical point. %Furthermore, we are interested in the properties of all the terms of the theta function in these cluster variables. Especially, we are interested in the log-concavity in sense of Okounkov \cite{Oko03}.

When the cluster algebra is of finite type, the theta functions are exactly cluster monomials \cite[Theorem 7.20]{GHKK18}. Particularly for the cluster algebras of type $A_n$, the theta basis is exactly the cluster monomial basis. Let us index the theta function $\theta_i$ by the tropical point $i$. Then
\begin{align}\theta_{i_1}\cdots \theta_{i_n}=\sum c_{i_1,\cdots,i_n}^j \theta_{j},\end{align}
where $c_{i_1,\cdots,i_n}^j\in \mathbb{Z}_{\geq 0}$.
In \cite{Shen14}, Shen proved that the support of the structure constants $c_{i_1,\cdots,i_n}^j$ is exactly a cluster convex hull of $i_1,\cdots,i_n$. Fix the cluster variables $x_1,\cdots,x_k$ of a seed. By \Cref{theorem:LE}, any theta function $\theta_c$ can be written as a polynomial $P_c$ of $x_1,\cdots,x_k$ diving the theta function $\theta_d=x_1^{j_1}\cdots x_k^{j_k}$ where $j_i\in \mathbb{Z}_{\geq 0}$. Thus, the support of the coefficients of $\theta_c$ or $P_c$ is a cluster convex hull of $c,d$. 
We say that the structure constants $c_{i_1,\cdots,i_n}^j$ are {\em log-concave} in $j$ if for any cluster chart $\mathbb{R}^k$ of $j$, there is a log-concave function $f(j)$ passing through all non-zero $c_{i_1,\cdots,i_n}^j$. 
To further understand these coefficients, we conjecture that
\begin{conjecture}\label{conj1}
For certain cluster algebra including all the finite type cluster algebra, the structure constants $c_{i_1,\cdots,i_n}^j$ of the theta basis are log-concave in $j$. Particularly, the coefficients of $\theta_c$ in the cluster variable of any given seed are log-concave.
\end{conjecture}
In this paper, we focus on studying the cluster algebras of type $A_n$. Recall that there is a geometric realization for them: triangulation. Afterwards, Schiffler \cite{Sch08} provided an expansion formula of cluster variables by $T$-path, see \Cref{Tpath}. Then, by use of the combinatorial and geometric information (intersection numbers) of $T$-path, we prove the main theorem as follows.
\begin{theorem}[Theorem \ref{main1}]
All the cluster variables of type $A_n$ are log-concave, that is the coefficient-free cluster algebras of type $A_n$ are log-concave.
\end{theorem}

For the cluster algebras of type $A_2$, we consider the corresponding cluster monomials. By use of binomial coefficients, we also prove the log-concavity of them as follows and give a conjecture for general cases, see \Cref{An conj}.
\begin{theorem}[Theorem \ref{main theorem}]
	The cluster monomials of type $A_2$ are log-concave.
\end{theorem}
In the end, according to the geometric realization of $d$-vectors and the relations between $d$-vectors and $f$-vectors, we show the log-concavity of $F$-polynomials of type $A_n$ as follows.
\begin{theorem}[\Cref{log-concavity of F poly}]
All the $F$-polynomials of type $A_n$ are log-concave.
\end{theorem}
It would also be very interesting to investigate Conjecture \ref{conj1} for the general decorated surface cases. 
%For $\operatorname{SL}_3$, following the initial work of Kuperberg on $\operatorname{SL}_3$-webs \cite{Kup96}, Sikora and Westbury \cite{SW07} found a reduced $\operatorname{SL}_3$-web basis for the $\operatorname{SL}_3$ character variety on the surface. Douglas and Sun \cite{DS20, DS24} indexed the reduced $\operatorname{SL}_3$-web basis by the Fock--Goncharov tropical coordinates. Then, Kim \cite{Kim20} showed that these tropical coordinates are exactly the highest term exponents of the quantum trace functions of the reduced $\operatorname{SL}_3$-webs. Afterwards, Shen, Sun and Weng \cite{SSW23} provided the geometric intersection number understanding of these tropical coordinates. It would be valuable to see the pattern of other terms of the trace function of these webs.

The paper is organized as follows: In Section 2, we recall the basic definitions of semifields, seeds, cluster algebras and cluster monomials. Then, we introduce the definition of log-concavity of Laurent polynomials (\Cref{log-concave2}). In Section 3, we review the geometric realization of cluster algebras of type $A_n$. In particular, by use of $T$-path, we prove the \Cref{coeff 012}. In Section 4, we prove the log-concavity of cluster variables of type $A_n$ based on the combinatorial and geometric information of $T$-path, see \Cref{main1}. In Section 5, by use of binomial coefficients, we prove that the cluster monomials of type $A_2$ are log-concave (\Cref{main theorem}). In addition, we give a conjecture about type $A_n(n\geq 3)$, see \Cref{An conj}. Finally, in Section 6, we recall the definitions of $F$-polynomials and $f$-vectors. Then, by use of the relations between $f$-vectors and $d$-vectors, we prove the log-concavity of $F$-polynomials of type $A_n$ (\Cref{log-concavity of F poly}).

\section*{Conventions}
\begin{itemize}[leftmargin=1em]\itemsep=0pt
\item In this paper, the integer ring, the set of non-negative integers, the set of positive integers and the rational number field are denoted by $\mbZ$, $\mbN$, $\mbN_{+}$ and $\mbQ$ respectively. 
\item We denote by $\text{Mat}_{n\times n}(\mbZ)$ the set of all $n\times n$ integer square matrices. An integer square matrix $B$ is called \emph{skew-symmetrizable} if there exists a positive integer diagonal matrix $D$ such that $DB$ is skew-symmetric and $D$ is called the \emph{left skew-symmetrizer} of $B$. 
\item For any $a\in \mbZ$, we denote $[a]_{+}=\max(a,0)$. For any $B=(b_{ij})_{n\times n}\in \text{Mat}_{n\times n}(\mbZ)$, we denote $[B]_{+}=([b_{ij}]_{+})_{n\times n}$.
%\item The set of Laurent polynomials of $x_{1},\dots,x_{m}$ over $\mbZ$ is $\mbZ[x_{1}^{\pm 1},\dots,x_{m}^{\pm 1}]$.
\item Let $B^{k\bullet}$ be the matrix obtained from $B$ by replacing all entries outside of the $k$-th row with zeros and $B^{\bullet k}$ be the matrix obtained from $B$ by replacing all entries outside of the $k$-th column with zeros. We denote $J_k$ the diagonal matrix obtained from the identity matrix by replacing the $k$-th diagonal entry with $-1$.

%\item For any $B=(b_{ij}),B^{\prime}=(b_{ij}^{\prime}) \in \text{Mat}_{n\times n}(\mbZ)$, we denote $\max(B,B^{\prime})=(m_{ij})$, where $m_{ij}=\max(b_{ij},b_{ij}^{\prime})$.
%\item Let $B^{k\bullet}$ be the matrix obtained from $B$ by replacing the entries in $i$-th rows $(i\neq k)$ with zeros. Let $B^{\bullet k}$ be the matrix obtained from $B$ by replacing the entries in $i$-th columns $(i\neq k)$ with zeros. We denote $J_k$ the diagonal matrix obtained from the identity matrix by replacing the $k$-th diagonal entry with $-1$.
\end{itemize}

\section*{Acknowledgment}
Z. Chen is supported by the CSC grant 202406340022 and Z. Sun is supported by the NSFC grant 12471068.

\section{Preliminaries}
In this section, we recall some basic but important notions about cluster algebras as introduced by Fomin-Zelevinsky \cite{FZ02,FZ03,FZ04}. In addition, we introduce the definition of log-concavity of Laurent polynomials and their properties.
\subsection{Semifields and cluster algebras}\

Firstly, in this subsection, we recall some definitions and properties of semifields and cluster algebras.
\begin{definition}[\emph{Semifield}]
A \emph{semifield} is a multiplicative abelian group $(\mbP,\cdot)$ which is equipped with an additive operation $\oplus$ such that for any $a,b,c \in \mbP$,
\begin{enumerate}[leftmargin=2em]
	\item $a \oplus b=b \oplus a$,
	\item $(a \oplus b)c = ac \oplus bc$,
	\item $(a \oplus b) \oplus c = a \oplus (b \oplus c).$
\end{enumerate}
Then, we denote it by $(\mbP, \cdot, \oplus)$.
\end{definition}
\begin{example} The following examples are two important semifields.
\begin{enumerate}[leftmargin=2em]
	\item Let $\mbP_{\text{triv}}=\{1\}$ be a trivial multiplicative group equipped with addition $\oplus$ such that $1\oplus 1=1$. Then, $\mbP_{\text{triv}}$ becomes a semifield and it is called a \emph{trivial semifield}.
	\item  Let $\mbP_{\text{trop}}=\text{Trop}(u_1,\dots, u_n)$ be a multiplicative abelian group freely generated by formal variables $u_1,\dots,u_n$ with addition $\oplus$ as follows:
\begin{align}
\prod_{i=1}^n u_i^{a_i} \oplus \prod_{i=1}^{n} u_i^{b_i}=\prod_{i=1}^{n} u_j^{\min(a_i,b_i)}.
\end{align}

\end{enumerate}
\end{example} 
Now, we fix $n\in \mbN_{+}$. According to \cite{FZ02}, the group ring $\mbZ\mbP$ is a domain and we can construct its fractional field, which is denoted by $\mbQ\mbP$. Let $\mcF$ be a rational function field with $n$ indeterminates over $\mbQ\mbP$ and we call it the \emph{ambient field}.
\begin{definition}[\emph{Seeds}]
	A \emph{labeled seed} is a triple $(\mfx,\mfy,B)$ such that $\mfx=(x_1,\dots,x_n)$ is an $n$-tuple of algebraically independent generating elements of $\mcF$, $\mfy=(y_1,\dots,y_{n})$ is an $n$-tuple of $\mbP$ and $B=(b_{ij})_{n\times n}$ is a skew-symmetrizable matrix. We call the $n$-tuple $\mfx$ \emph{cluster}, the element $x_i$ \emph{cluster variable}, the element $y_i$ \emph{coefficient variable} and $B$ the \emph{exchange matrix} respectively.
\end{definition}
Furthermore, there is an important notion of the mutation of a seed.
\begin{definition}[\emph{Seed mutations}] 
	Let $(\mfx,\mfy,B)$ be a labeled seed and $1\leq k\leq n$, we define a new labeled seed $\mu_k(\mfx,\mfy,B)=(\mfx^{\prime},\mfy^{\prime},B^{\prime})$ as follows: \begin{enumerate}[leftmargin=2em]
	\item $\mfx^{\prime}=(x_{1}^{\prime},\dots,x_{n}^{\prime})$, where \begin{align} \label{x-var}
		x_{i}^{\prime}=\left\{
		\begin{array}{ll}
			x_{k}^{-1}(\dfrac{y_{k}}{1 \oplus y_{k}}\prod\limits_{j=1}^{n}x_{j}^{[b_{jk}]_{+}}+\dfrac{1}{1 \oplus y_{k}}\prod\limits_{j=1}^{n}x_{j}^{[-b_{jk}]_{+}}), &   i=k, \\
			x_{i}, &   i \neq k. 
		\end{array} \right.
	\end{align}
\item $\mfy^{\prime}=(y_1^{\prime},\dots,y_{n}^{\prime})$, where 	\begin{align} \label{y-var}
		y_{i}^{\prime}=\left\{
		\begin{array}{ll}
			y_{k}^{-1}, &   i=k, \\
			y_{i}y_{k}^{[b_{ki}]_{+}}(1\oplus y_{k})^{-b_{ki}}, &   i \neq k.
		\end{array} \right.
	\end{align}
\item $B^{\prime}=(b_{ij}^{\prime})_{n\times n}$ is given by 
	\begin{align} \label{matrix mutation}
		b_{ij}^{\prime}=\left\{
		\begin{array}{ll}
			-b_{ij}, &   i=k \;\;\mbox{or}\;\; j=k, \\
			b_{ij}+[b_{ik}]_{+}b_{kj}+b_{ik}[-b_{kj}]_{+}, &   i\neq k \;\;
			\mbox{and}\; j\neq k. 
		\end{array} \right.
	\end{align}
\end{enumerate}
In fact, $(\mfx^{\prime},\mfy^{\prime},B^{\prime})$ is still a seed and $\mu_k$ is an involution, that is $\mu_k(\mfx^{\prime},\mfy^{\prime},B^{\prime})=(\mfx,\mfy,B)$, see \cite{Nak23}. Then, $(\mfx^{\prime},\mfy^{\prime},B^{\prime})$ is called the \emph{k-direction mutation} of $(\mfx,\mfy,B)$.
\end{definition}
\begin{remark}\label{equiv coeff}
	The coefficients were defined in \cite[Definition 5.3]{FZ02} and \cite[Section 1.2]{FZ03} as a $2n$-tuple $\mfp=(p_1^{\pm 1},\dots,p_n^{\pm 1})$ of $\mbP$, such that $p_i^{+}\oplus p_i^{-}=1$ for any $i\in \{1,\dots,n\}$. According to \cite[Formulas (5.2) \& (5.3)]{FZ02}, the two setups are equivalent  by setting 
	\begin{align}
		y_i=\dfrac{p_i^{+}}{p_i^{-}},
	\end{align}
	and $p_i^{\pm 1}$ can be recovered by 
	\begin{align}
		p_i^{+}=\dfrac{y_i}{y_i\oplus 1},\ p_i^{-}=\dfrac{1}{y_i\oplus 1}.
	\end{align}
\end{remark}
\begin{definition}[\emph{Cluster patterns}] 
	A \emph{cluster pattern} $\mathbf{\Sigma}=\{(\mfx_t,\mfy_t,B_t)|\ t\in \mbT_n\}$ is a collection of labeled seeds which are indexed by the vertices of $n$-regular tree $\mbT_{n}$, such that $(\mfx_{t^{\prime}},\mfy_{t^{\prime}},B_{t^{\prime}})=\mu_k(\mfx_t,\mfy_t,B_t)$ for any $t \stackrel{k}{\longleftrightarrow} t^{\prime}$ in $\mbT_n$. In the following, we use the notations that $$\mathbf{x}_{t}=(x_{1;t},\dots,x_{n;t}),\ \mathbf{y}_{t}=(y_{1;t},\dots,y_{n;t}),\  B_{t}=(b_{ij;t})_{n\times n}.$$ For an arbitrary fixed vertex $t_0\in \mbT_n$, we call the seed $(\mfx_{t_0},\mfy_{t_0},B_{t_0})$ \emph{initial seed} and denote the \emph{initial cluster} by $\mathbf{x}_{t_0}=\mfx=(x_{1},\dots,x_{n})$, the \emph{initial coefficients} by $\mathbf{y}_{t_0}=\mfy=(y_{1},\dots,y_{n})$ and the \emph{initial exchange matrix} by $B_{t_0}=B=(b_{ij})_{n\times n}$.  

\end{definition}
Now, we are ready to define the crucial notions of cluster algebras.
\begin{definition}[\emph{Cluster algebras}]
	For a cluster pattern $\mathbf{\Sigma}$, the \emph{cluster algebra} $\mcA=\mcA(\mathbf{\Sigma})$ is the $\mbZ\mbP$-subalgebra of $\mcF$ generated by all the cluster variables $\{x_{i;t}\mid i=1,\dots,n;t\in\mbT_{n}\}$. Here, $n$ is called the \emph{rank} of $\mcA$ or $\mathbf{\Sigma}$.
	\end{definition}
Furthermore, if $\mbP=\mbP_{\text{trop}}=\text{Trop}(x_{n+1},\dots, x_m)$, then $\mcA$ is called \emph{of geometric type}. 
If $\mbP=\mbP_{\text{triv}}$, then $\mcA$ is said to be \emph{coefficient-free} or \emph{with trivial coefficients}. In this case, the seed $(\mfx,\mfy,B)$ can be reduced to $(\mfx,B)$. Furthermore, $\mbZ\mbP$ and $\mbQ\mbP$ can be reduced to $\mbZ$ and $\mbQ$ respectively, that is the cluster algebra $\mcA$ is a $\mbZ$-subalgebra of $\mcF$ generated by all the cluster variables. 

\begin{theorem}[{\cite[Theorem 3.1]{FZ02},\cite[Theorem 4.10]{GHKK18}}]
\label{theorem:LE}
In a cluster algebra $\mcA$, any cluster variable is a Laurent polynomial in terms of the initial cluster with coefficients in $\mbN\mbP$.
	
\end{theorem}
More precisely, we have the Laurent expression of $x_{i;t}$ in terms of the initial cluster $\mfx=(x_1,\dots,x_n)$ as follows:\begin{align}
	x_{i;t}=\dfrac{N_{i;t}(x_1,\dots, x_n)}{x_1^{d_{1i;t}}\cdots x_n^{d_{ni;t}}},\  d_{ji;t}\in \mbZ, \label{express of d-vector}
\end{align} where $N_{i;t}(x_1,\dots, x_n)$ is a polynomial with coefficients in $\mbN\mbP$ which is not divisible by any $x_j$. The integer vector $\mfd_{i;t}=(d_{ji;t})_{j=1}^n$ is called the \emph{denominator vector} ($d$-vector) of $x_{i;t}$. In fact, the recurrence relations for $\mfd$-vectors are given as follows, cf. \cite[Section 4.3]{FZ04}. The initial conditions are $\mfd_{l;t_0}=-\mfe_{l}$ and the recursion formula is  
		\begin{align} \label{d-vector}
		\mfd_{l;t^{\prime}}=\left\{
		\begin{array}{ll}
			\mfd_{l;t}, &   l\neq k, \\
			-\mfd_{k;t}+\text{max}\ (\sum\limits_{i=1}^n[b_{ik;t}]_{+}\mfd_{i;t},\sum\limits_{i=1}^n[-b_{ik;t}]_{+}\mfd_{i;t}), &   l=k,
		\end{array} \right.
	\end{align}
	for $t \stackrel{k}{\longleftrightarrow} t^{\prime}$ in $\mbT_n$.
\begin{definition}[\emph{Finite type}]
A cluster algebra $\mcA$ is called \emph{of finite type} if it contains finitely many distinct seeds. Otherwise, it is called \emph{of infinite type}.
\end{definition} 
Let $B=(b_{ij})_{n\times n}$ be a skew-symmetrizable integer matrix whose \emph{Cartan counterpart} is a symmetrizable generalized Cartan matrix $A=A(B)=(a_{ij})_{n\times n}$, where 
\begin{align}
	a_{ij}=\left\{
		\begin{array}{ll}
			2, &   i=j, \\
			-|b_{ij}|, &   i \neq j.
		\end{array} \right.
\end{align}
\begin{theorem}[{\cite[Theorem 1.8]{FZ03}}] A cluster algebra is of finite type if and only if it contains an exchange matrix $B$ whose Cartan counterpart $A(B)$ is a Cartan matrix of finite type.
	
\end{theorem}
In particular, when the  cluster algebra $\mcA$ is of rank $2$, whose initial exchange matrix is given by \begin{align}
	\begin{pmatrix}
	0 & b \\ -c & 0
\end{pmatrix},
\end{align}
where  $b,c \in \mbN$. Then, it is of finite type if and only if $0\leq bc\leq 3$. If $bc=4$, the cluster algebra $\mcA$ is called of \emph{affine type}. If $bc\geq 5$, it is called of \emph{non-affine type}.
\begin{definition}[\emph{Dynkin type}]
	Let $X_n\ (\text{e.g.}\,A_n, B_n, \dots)$ be a Dynkin diagram with $n$ vertices. A cluster algebra $\mcA$ is called \emph{of type $X_n$} if one of its exchange matrices $B$ has Cartan counterpart of type $X_n$.
\end{definition}
\begin{remark}
	The cluster algebra $\mcA$ is of type $A_n, D_n, E_6, E_7$ or $E_8$ if and only if one of its exchange matrices $B$ corresponds to a quiver which is an orientation of a Dynkin diagram of the type above.
\end{remark}
There is another crucial notion of cluster monomials.
\begin{definition}[\emph{Cluster monomials}]\label{cluster monomials}
	A \emph{cluster monomial} is a product of non-negative powers of cluster variables which all belong to a same cluster. More precisely, for any $t\in \mbT_n$, the set of cluster monomials is $\{x_{1;t}^{m_1}\dots x_{n;t}^{m_n}\mid m_1,\dots,m_n\in \mbZ_{\geq 0}\}$.
\end{definition}
The notion of cluster monomials plays an important role in cluster theory. In \cite[Theorem 7.20]{GHKK18}, it was proved that for a cluster algebra, all distinct cluster monomials are linearly independent over $\mbZ$. In addition, the cluster monomials are contained in the \emph{theta functions}. When the cluster algebra is of finite type, theta functions are strictly cluster monomials. Here, for brevity, we do not recall the definition of theta functions.  
\begin{example}[\emph{Coefficient-free $A_2$ type}]\label{A2 type} Let $(\mfx,B)$ be the initial seed, where \begin{align}
	\mfx=(x_1,x_2),\ B=\begin{pmatrix}0 & 1\\ -1 & 0\end{pmatrix}.
\end{align}
	Note that the (labeled) clusters are 10-periodic and they are as follows:
$$\begin{array}{c}
(x_1,x_2),\ (\frac{x_2+1}{x_1},x_2),\ (\frac{x_2+1}{x_1},\frac{x_1+x_2+1}{x_1x_2}),\ (\frac{x_1+1}{x_2},\frac{x_1+x_2+1}{x_1x_2}),\ (\frac{x_1+1}{x_2},x_1),\\
(x_2,x_1),\ (x_2,\frac{x_2+1}{x_1}),\ (\frac{x_1+x_2+1}{x_1x_2},\frac{x_2+1}{x_1}),\ (\frac{x_1+x_2+1}{x_1x_2},\frac{x_1+1}{x_2}),\ (x_1,\frac{x_1+1}{x_2}).
\end{array}$$
Hence, there are five classes of cluster monomials: 
$$\begin{array}{l}
	x_1^{m_1}x_2^{m_2}, (\frac{x_2+1}{x_1})^{m_1}x_2^{m_2}, (\frac{x_2+1}{x_1})^{m_1}(\frac{x_1+x_2+1}{x_1x_2})^{m_2}, (\frac{x_1+1}{x_2})^{m_1}(\frac{x_1+x_2+1}{x_1x_2})^{m_2}, (\frac{x_1+1}{x_2})^{m_1}(x_1)^{m_2}, 
\end{array}$$
where $m_1,m_2 \in \mbZ_{\geq 0}$.
\end{example}
\begin{remark}
	Note that the cluster monomials of type $A_2$ are independent of the choices of the initial exchange matrix. That is, if the initial exchange matrix is $-B$, then the cluster monomials are same as above. However, the cases of type $A_n\ (n\geq 3)$ are quite different.
\end{remark}

\subsection{Log-concavity of Laurent polynomials}\

In this subsection, we introduce the notion of Laurent polynomials and its relation with concave functions.
\begin{definition}[\emph{Concave functions}] \label{concaveD}
Let $\Omega\subseteq \mbR^{m}$ be a convex set.
	A function $f(x_1,\dots,x_m): \Omega\rightarrow \mbR$ is said to be \emph{concave} if for any $\mfx, \mfy \in \Omega$ and $\lambda\in [0,1]$, then \begin{align}f((1-\lambda)\mfx+\lambda\mfy)\geq (1-\lambda)f(\mfx)+\lambda f(\mfy).\end{align}
\end{definition}
%\begin{definition}\label{log-concave}
%	A nonzero non-negative real coefficient polynomial with $m$ variables \begin{align}f(x_1,\dots,x_m)=\sum\limits_{i_1=0}^{n_1}\dots\sum\limits_{i_m=0}^{n_m}a_{i_1,\dots,i_m}x_1^{i_1}\dots x_m^{i_m} \label{m poly} \end{align} is called \emph{log-concave} if for any $1\leq j\leq m$ and $0\leq i_j\leq n_j$, \begin{align}a^2_{i_1,\dots,i_j,\dots,i_m}\geq a_{i_1,\dots,i_{j}-1,\dots,i_m}a_{i_1,\dots,i_{j}+1,\dots,i_m},\label{inequality}\end{align} where $a_{i_1,\dots,-1,\dots,i_m}=a_{i_1,\dots,n_{j}+1,\dots,i_m}=0$.
%\end{definition}

\begin{definition}[\emph{Log-concavity of Laurent polynomials}]\label{log-concave2}
	A nonzero Laurent polynomial with non-negative real coefficients and $m$ variables \begin{align}f(x_1,\dots,x_m)=\sum\limits_{i_1=l_1}^{n_1}\dots\sum\limits_{i_m=l_m}^{n_m}a_{i_1,\dots,i_m}x_1^{i_1}\dots x_m^{i_m} \label{Laurent log} 
	\end{align} is called \emph{log-concave} if for any $1\leq j\leq m$ and $l_j\leq i_j\leq n_j$, \begin{align}a^2_{i_1,\dots,i_j,\dots,i_m}\geq a_{i_1,\dots,i_{j}-1,\dots,i_m}a_{i_1,\dots,i_{j}+1,\dots,i_m},\label{Laurent inequality}\end{align} where setting $a_{i_1,\dots,l_j-1,\dots,i_m}=a_{i_1,\dots,n_{j}+1,\dots,i_m}=0$.
\end{definition}

\begin{example}\label{log-concave examples}\
There are several examples and counter-examples as follows.
	\begin{enumerate}[leftmargin=2em]
		\item The polynomial with one variable $f(x)=(x+1)^n$ is log-concave since the binomial coefficients $C_n^k=\frac{n!}{k!(n-k)!}$ satisfy that 
		\begin{align}
			\frac{(C_n^k)^2}{C_n^{k-1}C_n^{k+1}}=\frac{(n-k+1)(k+1)}{(n-k)k}> 1.
		\end{align} Note that $x+1$ is log-concave, but $x^2+1$ is not log-concave.
		\item The Laurent polynomial with three variables \begin{align}f(x_1,x_2,x_3)=\dfrac{1+2x_2+x_2^2+x_1x_3}{x_1x_2x_3}\end{align} is log-concave by direct calculation. In fact, it is a cluster variable of type $A_3$, see \Cref{T-example}.
	\end{enumerate}
\end{example}
\begin{remark}\label{same up}
	Note that any two Laurent polynomials which are same up to a Laurent monomial factor keep the same property of log-concavity. In particular, when the cluster algebra is coefficient-free, by  formula \eqref{express of d-vector}, $x_{i;t}$ is log-concave if and only if $N_{i;t}(x_1,\dots,x_n)$ is log-concave.
\end{remark}
For the nonzero Laurent polynomial \eqref{Laurent log}, let $I=\{(i_1,\dots,i_m)\in \mbZ^m|\ a_{i_1,\dots,i_m}>0 \}$ and then it is a non-empty set.
 %$I_j=\{0\leq i_j\leq n_j|\ a_{i_1,\dots,i_m}>0\}$ and $\widehat{I}_j=[\min(I_j),\max(I_j)]$, where $1\leq j\leq m$. Note that the interval $\widehat{I}_j$ may be an empty set or a singleton set. Then we replace it by   
If the three consecutive terms with respect to \eqref{Laurent inequality} all belong to $I$, we take the logarithm  and have \begin{align} \ln(a_{i_1,\dots,i_j,\dots,i_m})\geq \dfrac{1}{2}\ln(a_{i_1,\dots,i_{j}-1,\dots,i_m})+\dfrac{1}{2}\ln(a_{i_1,\dots,i_{j}+1,\dots,i_m}).\end{align} Hence, it is direct that the \Cref{log-concave2} can be defined equivalently by a concave function passing through some integer points as follows.
\begin{lemma}\label{equivalent log-concave}
	The nonzero Laurent polynomial with non-negative real coefficients \eqref{Laurent log} is log-concave if and only if there exists a function $g(x_1,\dots,x_n):\Omega \rightarrow \mbR_{> 0}$, where $\Omega$ is a convex set containing $I$ such that:
	\begin{enumerate}
		\item If $a_{i_1,\dots,i_j-1,\dots,i_m}a_{i_1,\dots,i_j+1,\dots,i_m}>0$, then $a_{i_1,\dots,i_j,\dots,i_m}>0$.
		\item $g(i_1,\dots,i_m)=a_{i_1,\dots,i_m}$, where $(i_1,\dots,i_m)\in I$.
		\item $G(\mfx)=\ln(g(x_1,\dots,x_m))$ is concave over $\Omega$.
	\end{enumerate}
%$\widehat{I}_1\times \dots \times \widehat{I}_m$.
\end{lemma} 

\begin{remark}
We give the above equivalent geometric definition as \Cref{concaveD} in order to match up with the definition for all the cases right before \Cref{conj1} and provide a visually geometric interpretation of log-concavity. However, in what follows, we always use the \Cref{concaveD}.
\end{remark}
\section{Geometric realization of cluster algebras of type $A_n$}
In this section, we recall some basic notions of geometric realization of cluster algebras of type $A_n$ based on \cite{FZ03}. We also recall a crucial expansion formula of cluster variables of type $A_n$ by use of $T$-path, which is given by Schiffler \cite{Sch08} and \cite{ST09}. 
\subsection{Triangulation of $(n+3)$-gons}\label{3.1}\

In the following, let $\mcP_{n+3}$ be a convex polygon with $n+3$ vertices. A \emph{boundary} is a line segment connecting two adjacent vertices of $\mcP_{n+3}$. A  \emph{diagonal} is a line segment connecting two non-adjacent vertices of $\mcP_{n+3}$. Two diagonals are called \emph{crossing} if they intersect in the interior of $\mcP_{n+3}$. A \emph{triangulation} $T$ of $\mcP_{n+3}$ is a maximal set of non-crossing
diagonals together with all the boundary edges. Hence, there are $n$ diagonals and $n+3$ boundary edges in a triangulation $T$. The boundary edges of $\mcP_{n+3}$ are denoted by $T_{n+1},\ldots,T_{2n+3}$.

According to \cite{FZ03}, in a cluster algebra of type $A_n$, the clusters are in bijection with triangulations of $\mcP_{n+3}$. For a
triangulation $T=\{T_1,\ldots,T_n,T_{n+1},\ldots,T_{2n+3}\}$, let
$\mfx=\mfx_{T}=\{x_1,\ldots,x_n\}$ be the corresponding initial cluster, where
we denote $x_i=x_{T_i} $. Moreover, take the semifield $\mbP=\text{Trop}(x_{n+1},\dots,x_{2n+3})$. Then, for any $1\leq k\leq n$, the $k$-direction mutation exchange relation is given by:
\begin{align}
	x_kx_{k^{\prime}}=x_ax_c+x_bx_d,
\end{align}
where $T_a,T_c$ are two
opposite edges of the quadrilateral and  $T_b,T_d$ are the other two
opposite edges. The mutation can also be described as a flip of diagonals in the quadrilateral and the relation is usually called \emph{ptolemy relation}, see \Cref{A flip}. We also associate an $n\times n$ matrix $B=B(T)=(b_{ij}(T))_{n\times n}$ with the triangulation $T$ as follows:
\begin{align}
	b_{ij}(T)=\left\{
		\begin{array}{ll}
			1, &   \text{If $T_i$ and $T_j$ belong to a same triangle,}\\ &\text{and the rotation from $T_i$  to $T_j$ is counter-clockwise}; \\
			-1, &   \text{If $T_i$ and $T_j$ belong to a same triangle,}\\ &\text{and the rotation from $T_i$  to $T_j$ is clockwise};\\
			0, & \text{If $T_i$ and $T_j$ do not belong to a same triangle.}
		\end{array} \right.
\end{align}
We can also define the element $b_{ij}$ for $n+1\leq i\leq 2n+3$ and $1\leq j\leq n$ as above. Then, the corresponding coefficient $2n$-tuple $\mfp=(p_1^+,p_1^-,\dots,p_n^+,p_n^-)$ is given by
\begin{align}p_i^+=\prod_{j\,\geq\, n+1 \,: \, b_{ji}=1 } x_j \qquad \textup{and} \qquad  
p_i^-=\prod_{j\,\geq\, n+1 \,: \, b_{ji}=-1} x_j.\end{align}

\begin{example}
	The zigzag triangulation of $6$-gon in \Cref{6-gon} corresponds to the initial seed $(\mfx,\mfp,B)$ of type $A_3$, where   
	\begin{align}
		\mfx=(x_1,x_2,x_3),\ B=\begin{pmatrix}0 & -1 & 0 \\ 1 & 0 & 1 \\ 0 & -1 & 0\end{pmatrix},
	\end{align}
	and $\mfp=(p_1^+,p_1^-,p_2^+,p_2^-,p_3^+,p_3^-)$ is as follows: \begin{align}
		(p_1^+,p_1^-)=(x_4,x_5x_9),\ (p_2^+,p_2^-)=(x_6x_9,1),\ (p_3^+,p_3^-)=(x_7,x_6x_8).
	\end{align}
Note that $p_i^+ \oplus p_i^-=1$ for any $i\in \{1,2,3\}$. Then by \Cref{equiv coeff}, the initial coefficient variables are \begin{align}\mfy=(y_1,y_2,y_3)=(x_4x_5^{-1}x_9^{-1},x_6x_9,x_7x_6^{-1}x_8^{-1}).\end{align}
\end{example}
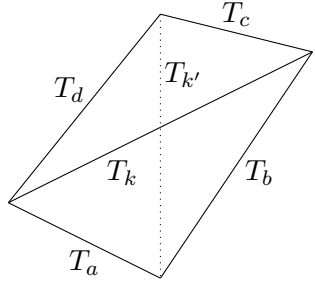
\begin{figure}
	\centering
	\begin{tikzpicture}
		\draw (0,0)--(2,-1);
		\draw (2,-1)--(4,2);
		\draw (0,0)--(2,2.5);
		\draw (2,2.5)--(4,2);
		\draw [dotted] (2,2.5)--(2,-1);
		\draw (0,0)--(4,2);
		\draw (1,-0.5) node[anchor=north]{$T_a$};
		\draw (3,2.8) node[anchor=north]{$T_c$};
		\draw (0.8,1.8) node[anchor=north]{$T_d$};
		\draw (3.3,0.7) node[anchor=north]{$T_b$};
		\draw (1.5,0.7) node[anchor=north]{$T_k$};
		\draw (2.3,2) node[anchor=north]{$T_{k^{\prime}}$};
	\end{tikzpicture}
	\caption{A flip in a quadrilateral.}	
\label{A flip}
\end{figure}

\subsection{$T$-paths}\ 

Let $T=\{T_1,\ldots,T_n,T_{n+1},\ldots,T_{2n+3}\}$ be an arbitrary
 triangulation of the convex $(n+3)$-polygon $\mcP_{n+3}$, where $T_1,\ldots,T_n$ are
 diagonals and $T_{n+1},\ldots,T_{2n+3}$ are boundaries. 
 Let $a$ and $b$ be two
non-adjacent vertices on the boundary and $X_{a,b}$ be the
diagonal between them.
\begin{definition}[\emph{$T$-paths}]\label{Tpath}
A \emph{$T$-path} $\mcP$ from $a$ to $b$ is  a sequence 
\[ \mcP = (v_0,v_1,\ldots,v_{\ell(\mcP)}\mid
i_1,i_2,\ldots,i_{\ell(\mcP)})\]
such that 
\begin{itemize}[leftmargin=3em]
\item[\textup{(T1)}]  $a=v_0,v_1,\dots,v_{\ell(\mcP)}=b$ are vertices of $\mcP$.
\item[\textup{(T2)}] $T_{i_k}$ connects the vertices $v_{{k-1}}$ and $v_{k}$
  for each $k=1,2,\dots,\ell(\mcP)$, where $T_{i_k}\in T$.
\item[\textup{(T3)}] If $j\ne k$, then $i_j\ne i_k$.
\item[\textup{(T4)}] $\ell(\mcP)$ is odd.
\item[\textup{(T5)}] If $k$ is even, then $T_{i_{k}}$ crosses $X_{a,b}$.
\item[\textup{(T6)}] If $j<k$ and both $T_{i_j}$ and $T_{i_{k}}$ cross
  $X_{a,b}$ then the crossing point of $T_{i_j}$  
and $X_{a,b}$ is closer to the vertex $ a $   
than the crossing point of $T_{i_{k}}$ and $X_{a,b}$.
\end{itemize}    
\end{definition}
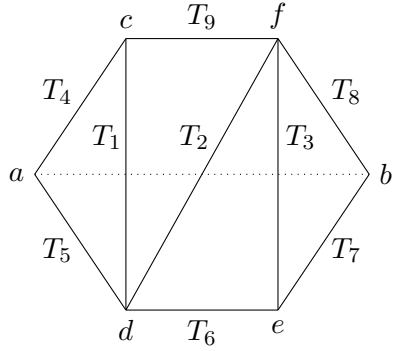
\begin{figure}
	\centering
	\begin{tikzpicture}
	\draw (0,0)--(2,0);
	\draw (-1.2,1.8)--(0,0);
	\draw (2,0)--(3.2,1.8);
	\draw (-1.2,1.8)--(0,3.6);
	\draw (2,3.6)--(3.2,1.8);
	\draw (0,3.6)--(2,3.6);
	\draw (0,3.6)--(0,0);
	%\draw (0,3.6)--(2,0);
	\draw (2,3.6)--(2,0);
	\draw (0,0)--(2,3.6);
	\draw [ dotted] (-1.2,1.8)--(3.2,1.8);
	\draw (-1.2,1.8) node[anchor=east]{$a$};
	\draw (3.2,1.8) node[anchor=west]{$b$};
	\draw (0,3.6) node[anchor=south]{$c$};
	\draw (0,0) node[anchor=north]{$d$};
	\draw (2,3.6) node[anchor=south]{$f$};
	\draw (2,0) node[anchor=north]{$e$};
	\draw (-0.25,2) node[anchor=south]{$T_1$};
	\draw (0.9,2) node[anchor=south]{$T_2$};
	\draw (2.3,2) node[anchor=south]{$T_3$};
	\draw (-0.9,2.6) node[anchor=south]{$T_4$};
	\draw (-0.9,0.5) node[anchor=south]{$T_5$};
	\draw (2.9,0.5) node[anchor=south]{$T_7$};
	\draw (2.9,2.6) node[anchor=south]{$T_8$};
	\draw (1,3.6) node[anchor=south]{$T_9$};
	\draw (1,0) node[anchor=north]{$T_6$};
	\end{tikzpicture}
\caption{A triangulation of $6$-gon (type $A_3$).}	
\label{6-gon}
\end{figure}
Let $\mcP_T(a,b)$ be the
set of all $T$-paths from $a$ to $b$. In the following, in any $T$-path $\mcP$ from $a$ to $b$, we classify the edges $T_{i_k}$ as follows:
\begin{itemize}[leftmargin=2em]
\itemsep=0pt
	\item $o$: The diagonals which are passed through at odd steps.
	\item $e$: The diagonals which are passed through at even  steps.
	\item $\widetilde{o}$: The boundaries which are passed through at odd steps.
	%\item $\widetilde{e}$: The boundaries which are passed through at even steps.
\end{itemize}
For example, we can see these notations in \Cref{Two cases}. For any $T$-path $\mcP$ from $a$ to $b$, 
we associate a Laurent monomial $x(\mcP)$ in the cluster algebra $\mcA$ as follows:
\begin{align}\label{T-express}
x(\mcP)=
{\prod_{k \textup{ odd}} x_{i_{k}}}
\
{\prod_{k \textup{ even}}x_{i_{k}}^{-1}}. 
\end{align}
Moreover, each cluster variable can be expressed by the sum of all such Laurent monomials $x(\mcP)$.
\begin{theorem}[{\cite[Theorem 1.2]{Sch08}}]
	Let $a$ and $b$ be two non-adjacent vertices of
  $\mcP_{n+3}$, $X_{a,b}$ be the diagonal between them and
  $x_{a,b}$ be the corresponding cluster variable. Then
\begin{align}
	x_{a,b}=\sum_{\mcP\in\mcP_T(a,b)} x(\mcP). \label{expression}
\end{align}
\end{theorem}
\begin{remark}\label{coe rmk} The formula \eqref{expression} also holds for the coefficient-free
  cluster algebras of type $A_n$  when setting
$x_i=1$, where $i=n+1, \dots,2n+3$.
\end{remark} 
\begin{example}\label{T-example} 
All the $T$-paths from $a$ to $b$ in \Cref{6-gon} are as follows, see \Cref{table1}. Hence, the coefficient-free cluster variable corresponding to the diagonal $ab$ in a cluster algebra of type $A_3$ is 
\begin{align}
	x_{a,b}=\dfrac{x_2^2+2x_2+1+x_1x_3}{x_1x_2x_3}.
\end{align}
It is clear that $x_{a,b}$ is log-concave with respect to each $x_i$.
\begin{table}[ht]
\centering
\scalebox{0.9}{
\begin{tabular}{|c|c|c|}
\hline 
&\\[-3mm]
$T$-paths $\mcP$ from $a$ to $b$ & Corresponding coefficient-free Laurent monomials $x(\mcP)$ \\ [1mm]
\hline
&\\[-3mm]
$(a,d,f,b\mid 5,2,8)$ & $\dfrac{1}{x_2}$\\ [3mm]
\hline 
&\\[-3mm]
$(a,c,d,e,f,b\mid 4,1,6,3,8)$ & $\dfrac{1}{x_1x_3}$\\ [3mm]
\hline 
&\\[-3mm]
$(a,d,c,f,e,b\mid 5,1,9,3,7)$ & $\dfrac{1}{x_1x_3}$\\ [3mm]
\hline 
&\\[-3mm]
$(a,c,d,f,e,b\mid 4,1,2,3,7)$ & $\dfrac{x_2}{x_1x_3}$\\[3mm]
\hline
&\\[-3mm] 
$(a,d,c,f,d,e,f,b\mid 5,1,9,2,6,3,8)$ & $\dfrac{1}{x_1x_2x_3}$\\ [3mm]
\hline
\end{tabular}}
\hspace{3cm}
\caption{$T$-paths from $a$ to $b$ in a cluster algebra of type $A_3$.}\label{table1}
\end{table}
%\begin{enumerate}
%	\item $(a,d,f,b\mid 5,2,8)$
%	\item $(a,c,d,e,f,b\mid 4,1,6,3,8)$
%	\item $(a,d,c,f,e,b\mid 5,1,9,3,7)$
%	\item $(a,c,d,f,e,b\mid 4,1,2,3,7)$
%	\item $(a,d,c,f,d,e,f,b\mid 5,1,9,2,6,3,8)$
%\end{enumerate}
\end{example}

\begin{lemma}[{\cite[Corollary 1.7]{Sch08}}] \label{01}
	Let $\mcA$ be a cluster algebra of type $A_n$ with coefficients $\mbP=\text{Trop}(x_{n+1},\dots,x_{2n+3})$ and $\mfx=\{x_1,\ldots,x_n\}$ be the initial cluster. Let
\begin{align}x_{i;t}=\dfrac{N_{i;t}(x_1,\dots, x_n)}{x_1^{d_{1i;t}}\cdots x_n^{d_{ni;t}}}=\frac{f_{i;t}(x_1,\ldots,x_n,x_{n+1},\ldots,x_{2n+3})}{x_1^{d_{1i;t}}\cdots x_n^{d_{ni;t}}}\end{align} be the Laurent
expression of $x_{i;t}$, where $f_{i;t}$ is a
polynomial which is not divisible by any of the $x_1,\ldots,x_n$.
Then, the coefficients of $f_{i;t}$ are  either $0$ or $1$.
\end{lemma} 
\begin{proposition}\label{coeff 012}
	Let $\mcA$ be a coefficient-free cluster algebra of type $A_n$. Then, the coefficients of $N_{i;t}(x_1,\dots, x_n)$ are $0$, $1$ or $2$.
	\end{proposition}
\begin{proof}
By \Cref{01}, it is direct  that the coefficients of $N_{i;t}$ can be $0$ and $1$. Particularly for the coefficient-free cluster algebras of type $A_n$, by \Cref{coe rmk} and the geometric realization, the boundaries in a $T$-path provide $1$ in \eqref{T-express}.  Hence, by the rules $(\text{T3})-(\text{T6})$ and the formula \eqref{T-express}, there are only two possible cases that two symmetric $T$-paths $\mcP_1$ and $\mcP_2$ passing by the same vertices generate the same element $x(\mcP_1)=x(\mcP_2)$, see \Cref{Two cases}. Then, the corresponding coefficient of $x(\mcP_1)$ in $N_{i;t}$ is $2$. However, note that the symmetric $T$-paths may not exist.  
	%More precisely, assume that there is another $T$-path $\mcP_3$ such that \begin{align}x(\mcP_3)=x(\mcP_1)=x(\mcP_2).\label{contra}\end{align} However, $\mcP_3$ must pass through two consecutive diagonals, which are not boundaries. Therefore, the conditions $(\text{T3})-(\text{T6})$ indicate that there is a contradiction with the equality \eqref{contra}.  
	%Moreover, the conditions $(\text{T3})-(\text{T6})$ also imply that there are not three or more symmetric $T$-paths which generate a same element $x(\mcP)$. 
	Therefore, by the cluster expansion formula \eqref{expression}, we obtain that the coefficients of $N_{i;t}$ are $0,1$ or $2$.
\end{proof}

\begin{remark}\label{prop-rmk}
Note that the \Cref{coeff 012} implies that if a $T$-path $\mcP$ from $a$ to $b$ passes through two  diagonals (not boundaries) consecutively, then $\mcP$ is the unique $T$-path from $a$ to $b$ providing the monomial $x(\mcP)$.
\end{remark}
\begin{figure}
\centering
\begin{tikzpicture}[scale=0.75]
\draw (-1,0) node[anchor=east]{Case $1$:};
\draw (0,0) node[anchor=east]{$a$};
\draw (7,0) node[anchor=west]{$b$};
\draw (1,1) node[anchor=south]{$A_1$};
\draw (1,-1) node[anchor=north]{$A_2$};
%\draw (2,-1.1) node[anchor=north]{$v_3$};
%\draw (2,1.1) node[anchor=south]{$v_4$};
%\draw (3,1.2) node[anchor=south]{$v_5$};

%\draw (4,1.2) node[anchor=south]{$v_{}$};
%\draw (5,1.1) node[anchor=south]{$v_{}$};
%\draw (4.8,-1.1) node[anchor=north]{$v_{\ell(\mcP)-3}$};
\draw (6.3,-1) node[anchor=north]{$A_{l-1}$};
\draw (6,0.9) node[anchor=south]{$A_l$};

\filldraw (0,0) circle (0.06);
\filldraw (7,0) circle (0.06);
	\draw (0,0)--(7,0);
	\draw[red, ->] (0,0)--(1,1);
	\draw[red,->] (1,1)--(1,-1);
	\draw[red,->] (1,-1)--(2,-1.1);
	\draw[red,->] (2,-1.1)--(2,1.1);
	\draw[red,->] (2,1.1)--(3,1.2);
	\draw[red,->] (3,1.2)--(3,-1.2);
	%\draw[->,dotted] (3,-1.2)--(4,-1.2);
	\draw[red,->] (4,-1.2)--(4,1.2);
	\draw[red,->] (4,1.2)--(5,1.1);
	\draw[red,->] (5,1.1)--(5,-1.1);
	\draw[red,->] (5,-1.1)--(6,-1);
	\draw[red,->] (6,-1)--(6,1);
	\draw[red,->] (6,1)--(7,0);
\draw (0,0.6) node[anchor= west] {$\textcolor{red}{\tilde{o}}$};
\draw (0.5,-0.3) node[anchor= west] {$\textcolor{red}{e}$};
\draw (1.2,-1.3) node[anchor= west] {$\textcolor{red}{\tilde{o}}$};
\draw (1.9,-0.3) node[anchor= west] {$\textcolor{red}{e}$};
\draw (2.2,1.4) node[anchor= west] {$\textcolor{red}{\tilde{o}}$};
\draw (2.9,-0.3) node[anchor= west] {$\textcolor{red}{e}$};
\draw (3.9,-0.3) node[anchor= west] {$\textcolor{red}{e}$};
\draw (4.9,-0.3) node[anchor= west] {$\textcolor{red}{e}$};
\draw (5.9,-0.3) node[anchor= west] {$\textcolor{red}{e}$};
\draw (3.1,-1.1) node[anchor= west] {$\textcolor{red}{\dots}$};
\draw (4.2,1.4) node[anchor= west] {$\textcolor{red}{\tilde{o}}$};
\draw (5.2,-1.3) node[anchor= west] {$\textcolor{red}{\tilde{o}}$};
\draw (6.4,0.7) node[anchor= west] {$\textcolor{red}{\tilde{o}}$};
	
\draw (9,1) node[anchor=south]{$A_1$};
\draw (9,-1) node[anchor=north]{$A_2$};
\draw (14,-1) node[anchor=north]{$A_{l-1}$};
\draw (14,0.9) node[anchor=south]{$A_l$};
\draw (8,0) node[anchor=east]{$a$};
\draw (15,0) node[anchor=west]{$b$};
\filldraw (8,0) circle (0.06);
\filldraw (15,0) circle (0.06);
	\draw (8,0)--(15,0);
	\draw[red, ->] (8,0)--(9,-1);
	\draw[red, ->] (9,-1)--(9,1);
	\draw[red, ->] (9,1)--(10,1.1);
	\draw[red, ->] (10,1.1)--(10,-1.1);
	\draw[red, ->] (10,-1.1)--(11,-1.2);
	\draw[red, ->] (11,-1.2)--(11,1.2);
	%\draw[->,dotted] (11,1.2)--(12,1.2);
	\draw[red, ->] (12,1.2)--(12,-1.2);
	\draw[red, ->] (12,-1.2)--(13,-1.1);
	\draw[red, ->] (13,-1.1)--(13,1.1);
	\draw[red, ->] (13,1.1)--(14,1);
	\draw[red, ->] (14,1)--(14,-1);
	\draw[red, ->] (14,-1)--(15,0);
\draw (8,-0.6) node[anchor= west] {$\textcolor{red}{\tilde{o}}$};
\draw (8.5,0.3) node[anchor= west] {$\textcolor{red}{e}$};
\draw (9.2,1.3) node[anchor= west] {$\textcolor{red}{\tilde{o}}$};	
\draw (9.9,0.3) node[anchor= west] {$\textcolor{red}{e}$};
\draw (10.2,-1.4) node[anchor= west] {$\textcolor{red}{\tilde{o}}$};
\draw (10.9,0.3) node[anchor= west] {$\textcolor{red}{e}$};
\draw (11.1,1.1) node[anchor= west] {$\textcolor{red}{\dots}$};
\draw (12.2,-1.4) node[anchor= west] {$\textcolor{red}{\tilde{o}}$};
\draw (11.9,0.3) node[anchor= west] {$\textcolor{red}{e}$};
\draw (12.9,0.3) node[anchor= west] {$\textcolor{red}{e}$};
\draw (13.9,0.3) node[anchor= west] {$\textcolor{red}{e}$};
\draw (13.2,1.3) node[anchor= west] {$\textcolor{red}{\tilde{o}}$};
\draw (14.4,-0.6) node[anchor= west] {$\textcolor{red}{\tilde{o}}$};

\draw (-1,-4) node[anchor=east]{Case $2$:};
\draw (0,-4) node[anchor=east]{$a$};
\draw (6.2,-4) node[anchor=west]{$b$};
\draw (1,-3) node[anchor=south]{$B_1$};
\draw (1,-5) node[anchor=north]{$B_2$};
\draw (5.3,-3) node[anchor=south]{$B_{l-1}$};
\draw (5.1,-5) node[anchor=north]{$B_l$};
\filldraw (0,-4) circle (0.06);
\filldraw (6.2,-4) circle (0.06);
	\draw (0,-4)--(6.2,-4);
	\draw[blue, ->] (0,-4)--(1,-3);
	\draw[blue, ->] (1,-3)--(1,-5);
	\draw[blue, ->] (1,-5)--(2,-5.1);
	\draw[blue, ->] (2,-5.1)--(2,-2.9);
	\draw[blue, ->] (2,-2.9)--(3,-2.8);
	\draw[blue, ->] (3,-2.8)--(3,-5.2);
%	\draw[->, dotted] (3,-5.2)--(4,-5.2);
	\draw[blue, ->] (4,-5.2)--(4,-2.8);
	\draw[blue, ->] (4,-2.8)--(5,-2.9);
	\draw[blue, ->] (5,-2.9)--(5,-5.1);
	\draw[blue, ->] (5,-5.1)--(6.2,-4);
\draw (0,-3.4) node[anchor= west] {$\textcolor{blue}{\tilde{o}}$};
\draw (0.5,-4.4) node[anchor= west] {$\textcolor{blue}{e}$};
\draw (1.9,-4.4) node[anchor= west] {$\textcolor{blue}{e}$};
\draw (2.9,-4.4) node[anchor= west] {$\textcolor{blue}{e}$};
\draw (3.9,-4.4) node[anchor= west] {$\textcolor{blue}{e}$};
\draw (4.9,-4.4) node[anchor= west] {$\textcolor{blue}{e}$};
\draw (1.2,-5.3) node[anchor= west] {$\textcolor{blue}{\tilde{o}}$};
\draw (2.2,-2.6) node[anchor= west] {$\textcolor{blue}{\tilde{o}}$};
\draw (4.2,-2.6) node[anchor= west] {$\textcolor{blue}{\tilde{o}}$};
\draw (5.5,-4.7) node[anchor= west] {$\textcolor{blue}{\tilde{o}}$};
\draw (3.1,-5.1) node[anchor= west] {$\textcolor{blue}{\dots}$};

\draw (8.5,-4) node[anchor=east]{$a$};
\draw (14.7,-4) node[anchor=west]{$b$};
\draw (9.5,-3) node[anchor=south]{$B_1$};
\draw (9.5,-5) node[anchor=north]{$B_2$};
\draw (13.5,-3) node[anchor=south]{$B_{l-1}$};
\draw (13.6,-5) node[anchor=north]{$B_l$};
\filldraw (8.5,-4) circle (0.06);
\filldraw (14.7,-4) circle (0.06);
	\draw (8.5,-4)--(14.7,-4);
	\draw[blue, ->] (8.5,-4)--(9.5,-5);
	\draw[blue, ->] (9.5,-5)--(9.5,-3);
	\draw[blue, ->] (9.5,-3)--(9.5,-3);
	\draw[blue, ->] (9.5,-3)--(10.5,-2.9);
	\draw[blue, ->] (10.5,-2.9)--(10.5,-5.1);
	\draw[blue, ->] (10.5,-5.1)--(11.5,-5.2);
	\draw[blue, ->] (11.5,-5.2)--(11.5,-2.8);
%	\draw[->, dotted] (11.5,-2.8)--(12.5,-2.8);
	\draw[blue, ->] (12.5,-2.8)--(12.5,-5.2);
	\draw[blue, ->] (12.5,-5.2)--(13.5,-5.1);
	\draw[blue, ->] (13.5,-5.1)--(13.5,-2.9);
	\draw[blue, ->] (13.5,-2.9)--(14.7,-4);
\draw (8.5,-4.6) node[anchor= west] {$\textcolor{blue}{\tilde{o}}$};
\draw (9,-3.6) node[anchor= west] {$\textcolor{blue}{e}$};
\draw (10.4,-3.6) node[anchor= west] {$\textcolor{blue}{e}$};
\draw (9.7,-2.7) node[anchor= west] {$\textcolor{blue}{\tilde{o}}$};
\draw (11.4,-3.6) node[anchor= west] {$\textcolor{blue}{e}$};
\draw (10.7,-5.4) node[anchor= west] {$\textcolor{blue}{\tilde{o}}$};
\draw (12.4,-3.6) node[anchor= west] {$\textcolor{blue}{e}$};
\draw (12.7,-5.4) node[anchor= west] {$\textcolor{blue}{\tilde{o}}$};
\draw (13.4,-3.6) node[anchor= west] {$\textcolor{blue}{e}$};
\draw (14,-3.3) node[anchor= west] {$\textcolor{blue}{\tilde{o}}$};
\draw (11.55,-2.9) node[anchor= west] {$\textcolor{blue}{\dots}$};
\end{tikzpicture}
\caption{Two possible cases of two symmetric $T$-paths.}
\label{Two cases}
\end{figure}

\section{Log-concavity of cluster algebras of type $A_n$}
In this section, we define the log-concavity of coefficient-free cluster algebras. Then, we prove the log-concavity of cluster variables of type $A_n$ with $n\geq 2$.
\begin{definition}[\emph{Log-concavity of cluster algebras}]\label{log-concavity of cluster}
	For a coefficient-free cluster algebra $\mcA$, the cluster variable $x_{i;t}$ is called \emph{log-concave} if it is log-concave as a Laurent polynomial by \eqref{express of d-vector}. If all the cluster variables are log-concave, then the cluster algebra $\mcA$ is called \emph{log-concave}.
\end{definition}
Now, we give the main theorem to exhibit a novel phenomenon: log-concavity of the cluster variables of type $A_n$. Beforehand, the sketch of the proof is as follows. \begin{enumerate}[leftmargin=2em] 
\item Fix the exponents of every cluster variables of type $A_n$ based on Laurent phenomenon.
\item Use the $T$-paths and the intersection information in a given triangulation to characterize the cluster variables.
\item Determine the coefficients of the three successive cluster variables with respect to log-concavity.
\end{enumerate} 
\begin{theorem}\label{main1}
	All the cluster variables of type $A_n$ are log-concave, that is the coefficient-free cluster algebras of type $A_n$ are log-concave. 
\end{theorem}
\begin{proof}  In the following, we take three steps to prove this theorem. Let $\mfx=(x_1,\dots,x_n)$ be the initial cluster and $T$ be the corresponding triangulation.\

\textbf{Step 1:} Note that by the cluster expansion formulas \eqref{T-express} and \eqref{expression} of type $A_n$, the degree of each $x_j$ in each Laurent monomial term of \eqref{express of d-vector} is $-1,0$ or $1$. Hence, to prove that each cluster variable is log-concave, it is sufficient to prove that for any $1\leq j\leq m$, the following inequality holds:
\begin{align}\label{inequality}
	a^2_{i_1,\dots,0,\dots,i_m}\geq a_{i_1,\dots,-1,\dots,i_m}a_{i_1,\dots,1,\dots,i_m},
\end{align}
where $-1,0,1$ are at the $j$-th position. Moreover, we only need to focus on the case that $a_{i_1,\dots,-1,\dots,i_m}a_{i_1,\dots,1,\dots,i_m}\neq 0$, which means there exists (at least one) $T$-paths $\mcP_1$ and $\mcP_2$ simultaneously such that 
\begin{align}
	x(\mcP_1)=x_1^{i_1}\dots x_j^{-1}\dots x_m^{i_m}, \ x(\mcP_2)=x_1^{i_1}\dots x_j\dots x_m^{i_m}. \label{two x(P)}
\end{align}
In the following figures, we use the red broken line with vertices $R_i$ to represent the $T$-path $\mcP_1$ passing through the diagonal $T_j$ at even step (corresponding to $x_j^{-1}$ in \eqref{two x(P)}) and the blue broken line with vertices $B_i$ to represent the $T$-path $\mcP_2$ passing through the diagonal $T_j$ at odd step (corresponding to $x_j$ in \eqref{two x(P)}).
\begin{figure}
\centering
\begin{tikzpicture}[scale=0.8]
	\draw [line width=2pt](0,0)--(11,3.3);
	\draw (0,0) node[anchor=east]{$a$};
	\draw (11,3.3) node[anchor=west]{$b$};
	\draw[blue][->] (6,0.5)--(7,4);
	\draw[blue][->] (4.5,3)--(6,0.5);
	\draw[red][->] (5.9,0.5)--(6.9,4);
	%\draw[red][->] (4.42,3)--(5.93,0.5);
	\filldraw (0,0) circle (0.06);
	\filldraw (11,3.3) circle (0.06);
	\filldraw (7,4.1) circle (0.06);
	\filldraw (5.97,0.43) circle (0.06);
	%\filldraw (4,0.7) circle (0.06);
	\filldraw (4.5,3) circle (0.06);
	\draw (7,4) node[anchor=south west] {$\textcolor{red}{R_0}$=$\textcolor{blue}{B_0}$};
	\draw (6,0.5) node[anchor=north ] {$\textcolor{red}{R_{-1}}$=$\textcolor{blue}{B_{-1}}$};
	\draw (3.8,3) node[anchor=south west] {$\textcolor{blue}{B_{-2}}$};
	%\draw (4,0.7) node[anchor=north east] {$\textcolor{red}{R_{-2}}$};
	\draw (6.7,2.8) node[anchor= east] {$\textcolor{red}{e}$};
	\draw (6.53,2.8) node[anchor= west] {$\textcolor{blue}{o}$};
	%\draw (4.8,2.3) node[anchor= north] {$\textcolor{red}{o}$};
	\draw (4.9,2.2) node[anchor= west] {$\textcolor{blue}{e}$};
	\draw (7.1,3.6) node[anchor= north] {$T_j$};
\end{tikzpicture}
\caption{Same orientation through $T_j$.}
\label{Same orientation}
\end{figure}
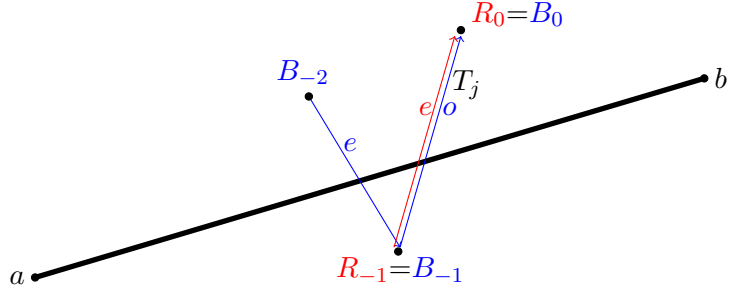
\begin{figure}
\centering
\begin{tikzpicture}[scale=0.8]
	\draw [line width=2pt](0,0)--(11,3.3);
	\draw (0,0) node[anchor=east]{$a$};
	\draw (11,3.3) node[anchor=west]{$b$};
	\draw[blue][->] (6,0.5)--(7,4);
	\draw[blue][->] (4.5,3)--(6,0.5);
	\draw[red][->] (5.9,0.5)--(6.9,4);
	\draw[red][->] (4.42,3)--(5.93,0.5);
	\filldraw (0,0) circle (0.06);
	\filldraw (11,3.3) circle (0.06);
	\filldraw (7,4.1) circle (0.06);
	\filldraw (5.97,0.43) circle (0.06);
	%\filldraw (4,0.7) circle (0.06);
	\filldraw (4.5,3) circle (0.06);
	\draw (7,4) node[anchor=south west] {$\textcolor{red}{R_0}$=$\textcolor{blue}{B_0}$};
	\draw (6,0.5) node[anchor=north ] {$\textcolor{red}{R_{-1}}$=$\textcolor{blue}{B_{-1}}$};
	\draw (3.2,3) node[anchor=south west] {$\textcolor{red}{R_{-2}}$=$\textcolor{blue}{B_{-2}}$};
	%\draw (4,0.7) node[anchor=north east] {$\textcolor{red}{R_{-2}}$};
	\draw (6.7,2.8) node[anchor= east] {$\textcolor{red}{e}$};
	\draw (6.53,2.8) node[anchor= west] {$\textcolor{blue}{o}$};
	\draw (4.8,2.3) node[anchor= north] {$\textcolor{red}{o}$};
	\draw (4.9,2.2) node[anchor= west] {$\textcolor{blue}{e}$};
	\draw (7.1,3.6) node[anchor= north] {$T_j$};
\end{tikzpicture}
\caption{Same orientation through $T_j$: case that $R_{-2}=B_{-2}$.}
\label{R_{-2}=B_{-2}}
\end{figure}
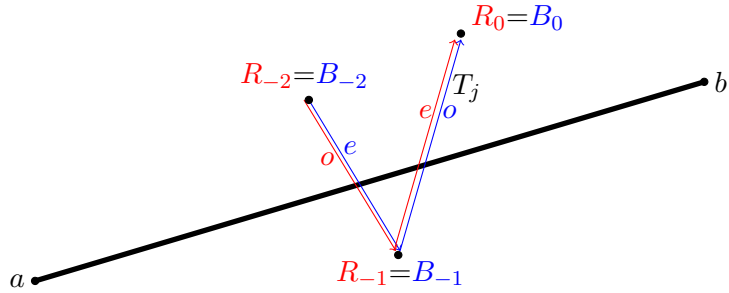
\

	\textbf{Step 2:} There are two possible orientations that $T$-paths $\mcP_1$ and $\mcP_2$ pass through the diagonal $T_j$. Now, we claim that 
	$\mcP_1$ and $\mcP_2$ pass through $T_j$ in the opposite orientation. Otherwise, we assume that their orientations are same such that \begin{align}R_{-1}R_0=B_{-1}B_0=T_j,\end{align} where the vertices satisfy $R_0=B_0$ and $R_{-1}=B_{-1}$. By $(\text{T5})$, there is a diagonal $B_{-2}B_{-1}$ in the given triangulation $T$ which is crossed by $ab$ and passed through by $\mcP_2$ in the even step, see \Cref{Same orientation}. Then, there are several choices as follows for the positions of the vertex $R_{-2}$ in the $T$-path $\mcP_1$. Now, we use the notation $x_{R_{i-1}R_i}$ or $x_{B_{i-1}B_i}$ to represent the initial cluster variable corresponding to the edge $R_{i-1}R_i$ or $B_{i-1}B_i$.
	\begin{enumerate}[leftmargin=2em]
		\item If $R_{-2}=B_{-2}$ (see \Cref{R_{-2}=B_{-2}}), then there are at least two different items in \eqref{two x(P)} as follows:		\begin{align}
			x(\mcP_1)=\dots x_{R_{-2}R_{-1}}x_j^{-1}\dots,\ x(\mcP_2)=\dots x^{-1}_{B_{-2}B_{-1}}x_j\dots,
		\end{align}
		where $x_{R_{-2}R_{-1}}=x_{B_{-2}B_{-1}}$. Since we focus on the case that there is only one different item about $x_j$ in \eqref{two x(P)}, we exclude this case.
		\item If $R_{-2}\neq B_{-2}$ and the edge $R_{-2}R_{-1}$ is not a boundary, then there are two possible cases, see \Cref{$R_{-2}R_{-1}$ is diagonal}. In the first case, by $(\text{T3})-(\text{T6})$, the term $x^{-1}_{B_{-2}B_{-1}}$ in $x(\mcP_2)$ does not appear in $x(\mcP_1)$. Similarly, in the second case, the term $x_{R_{-2}R_{-1}}$ in $x(\mcP_1)$ does not appear in $x(\mcP_2)$. Then, we exclude this case for the same reason as above.
  \begin{figure}
\centering
\begin{tikzpicture}[scale=0.5]
	\draw [line width=2pt](0,0)--(11,3.3);
	\draw (0,0) node[anchor=east]{$a$};
	\draw (11,3.3) node[anchor=west]{$b$};
	\draw[blue][->] (6,0.5)--(7,4);
	\draw[blue][->] (4.5,3)--(6,0.5);
	\draw[red][->] (5.9,0.5)--(6.9,4);
	\draw[red][->] (4,0.7)--(5.9,0.5);
	\filldraw (0,0) circle (0.06);
	\filldraw (11,3.3) circle (0.06);
	\filldraw (7,4.1) circle (0.06);
	\filldraw (5.97,0.43) circle (0.06);
	\filldraw (4,0.7) circle (0.06);
	\filldraw (4.5,3) circle (0.06);
	\draw (5.55,4) node[anchor=south west] {$\textcolor{red}{R_0}$=$\textcolor{blue}{B_0}$};
	\draw (6,0.5) node[anchor=north ] {$\textcolor{red}{R_{-1}}$=$\textcolor{blue}{B_{-1}}$};
	\draw (3.5,3) node[anchor=south west] {$\textcolor{blue}{B_{-2}}$};
	\draw (4,0.7) node[anchor=north east] {$\textcolor{red}{R_{-2}}$};
	\draw (6.7,2.8) node[anchor= east] {$\textcolor{red}{e}$};
	\draw (6.53,2.8) node[anchor= west] {$\textcolor{blue}{o}$};
	\draw (4.9,1.3) node[anchor= north] {$\textcolor{red}{o}$};
	\draw (4.85,2.2) node[anchor= west] {$\textcolor{blue}{e}$};
	\draw (6.8,2) node[anchor= north] {$T_j$};
	\draw [line width=2pt](12,0)--(23,3.3);
	\draw (0,0) node[anchor=east]{$a$};
	\draw (23,3.3) node[anchor=west]{$b$};
	\draw[blue][->] (18,0.5)--(19,4);
	\draw[blue][->] (16.5,3)--(18,0.5);
	\draw[red][->] (17.9,0.5)--(18.9,4);
	\draw[red][->] (17.6,4)--(17.9,0.5);
	\filldraw (12,0) circle (0.06);
	\filldraw (23,3.3) circle (0.06);
	\filldraw (19,4.1) circle (0.06);
	\filldraw (17.97,0.43) circle (0.06);
	\filldraw (17.6,4) circle (0.06);
	\filldraw (16.5,3) circle (0.06);
	\draw (18.4,4) node[anchor=south west] {$\textcolor{red}{R_0}$=$\textcolor{blue}{B_0}$};
	\draw (18,0.5) node[anchor=north ] {$\textcolor{red}{R_{-1}}$=$\textcolor{blue}{B_{-1}}$};
	\draw (15.5,3) node[anchor=south west] {$\textcolor{blue}{B_{-2}}$};
	\draw (17.6,4) node[anchor=south ] {$\textcolor{red}{R_{-2}}$};
	\draw (18.7,2.8) node[anchor= east] {$\textcolor{red}{e}$};
	\draw (18.53,2.8) node[anchor= west] {$\textcolor{blue}{o}$};
	\draw (17.45,3) node[anchor= north] {$\textcolor{red}{o}$};
	\draw (16.2,2.2) node[anchor= west] {$\textcolor{blue}{e}$};
\draw (18.8,2) node[anchor= north] {$T_j$};
\end{tikzpicture}
\caption{Two cases that $R_{-2}R_{-1}$ are diagonals with $R_{-2}\neq B_{-2}$.}
\label{$R_{-2}R_{-1}$ is diagonal}
\end{figure}
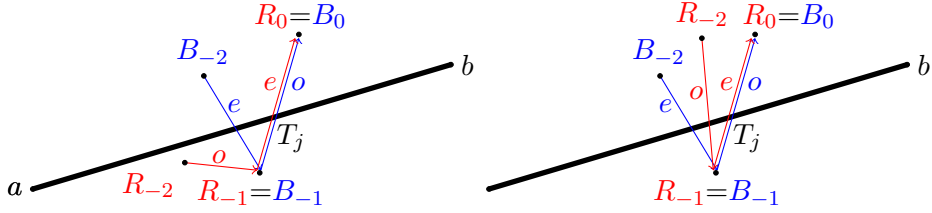
		\item If $R_{-2}\neq B_{-2}$ and the edge $R_{-2}R_{-1}$ is a boundary, see \Cref{$R_{-2}R_{-1}$ is boundary}. By $(\text{T3})-(\text{T6})$, the term $x^{-1}_{B_{-2}B_{-1}}$ in $x(\mcP_2)$ does not appear in $x(\mcP_1)$. Hence, we  exclude this case as well.
\begin{figure}
\centering
\begin{tikzpicture}[scale=0.8]
	\draw [line width=2pt](0,0)--(11,3.3);
	\draw (0,0) node[anchor=east]{$a$};
	\draw (11,3.3) node[anchor=west]{$b$};
	\draw[blue][->] (6,0.5)--(7,4);
	\draw[blue][->] (4.5,3)--(6,0.5);
	\draw[red][->] (5.9,0.5)--(6.9,4);
	\draw[red][->] (4,0.7)--(5.9,0.5);
	\filldraw (0,0) circle (0.06);
	\filldraw (11,3.3) circle (0.06);
	\filldraw (7,4.1) circle (0.06);
	\filldraw (5.97,0.43) circle (0.06);
	\filldraw (4,0.7) circle (0.06);
	\filldraw (4.5,3) circle (0.06);
	\draw (7,4) node[anchor=south west] {$\textcolor{red}{R_0}$=$\textcolor{blue}{B_0}$};
	\draw (6,0.5) node[anchor=north ] {$\textcolor{red}{R_{-1}}$=$\textcolor{blue}{B_{-1}}$};
	\draw (4,3) node[anchor=south west] {$\textcolor{blue}{B_{-2}}$};
	\draw (4,0.7) node[anchor=north east] {$\textcolor{red}{R_{-2}}$};
	\draw (6.7,2.8) node[anchor= east] {$\textcolor{red}{e}$};
	\draw (6.53,2.8) node[anchor= west] {$\textcolor{blue}{o}$};
	\draw (4.8,1.2) node[anchor= north] {$\textcolor{red}{\widetilde{o}}$};
	\draw (4.5,2.2) node[anchor= west] {$\textcolor{blue}{e}$};
	\draw (7.1,3.6) node[anchor= north] {$T_j$};
\end{tikzpicture}
\caption{Same orientation through $T_j$: case that $R_{-2}R_{-1}$ is a boundary.}
\label{$R_{-2}R_{-1}$ is boundary}
\end{figure}
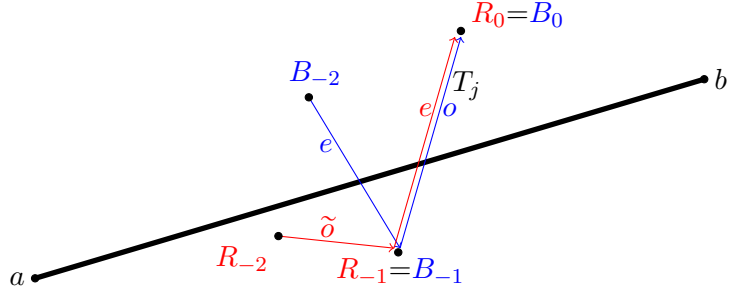  
		\end{enumerate}
Consequently, the $T$-paths $\mcP_1$ and $\mcP_2$ satisfying \eqref{two x(P)} must pass through the diagonal $T_j$ in the opposite orientation.

	\textbf{Step 3:} Finally, we focus on proving that if $a_{i_1,\dots,-1,\dots,i_m}a_{i_1,\dots,1,\dots,i_m}\neq 0$, then $a_{i_1,\dots,0,\dots,i_m}=2$. Note that the edge $R_{-2}R_{-1}$ does not cross the edge $B_{-2}B_{-1}$. Therefore, there are several possible choices for the positions of the vertex $R_{-2}$.
	\begin{enumerate}[leftmargin=2em]
		\item If $R_{-2}\neq B_{-2}$ and it lies between the vertices $B_{-2}$ and $R_{-1}$, see \Cref{RB}. Note that $\mcP_{1}$ passes through $R_{-3}R_{-2}$ in the even step. Hence, by $(\text{T5})$ and the non-crossing of diagonals in a triangulation, we get \begin{align}R_{-3}=R_{0}=B_{-1}.\end{align} However, no matter whether $R_{-2}R_{-1}$ is a diagonal or a boundary, the rules $(\text{T3})$ to $(\text{T6})$ guarantee that the term $x^{-1}_{R_{-3}R_{-2}}$ in $x(\mcP_1)$ does not appear in $x(\mcP_2)$. Hence, we exclude this case.
  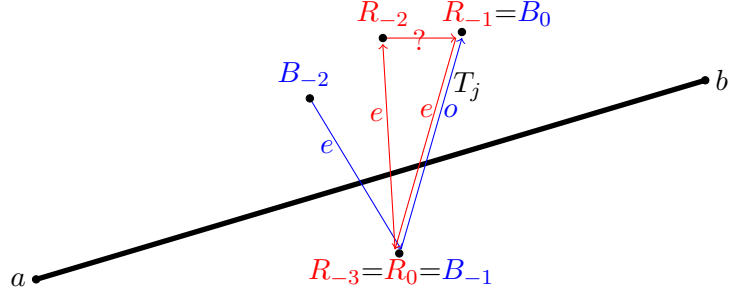
\begin{figure}
\centering
\begin{tikzpicture}[scale=0.8]
	\draw [line width=2pt](0,0)--(11,3.3);
	\draw (0,0) node[anchor=east]{$a$};
	\draw (11,3.3) node[anchor=west]{$b$};
	\draw[blue][->] (6,0.5)--(7,4);
	\draw[blue][->] (4.5,3)--(6,0.5);
	\draw[red][->] (6.9,4)--(5.9,0.5);
	\draw[red][->] (5.7,4)--(6.9,4);
	\draw[red][->] (5.9,0.5)--(5.7,3.9);
	\filldraw (0,0) circle (0.06);
	\filldraw (11,3.3) circle (0.06);
	\filldraw (7,4.1) circle (0.06);
	\filldraw (5.97,0.43) circle (0.06);
	\filldraw (5.7,4) circle (0.06);
	\filldraw (4.5,3) circle (0.06);
	\draw (6.5,4) node[anchor=south west] {$\textcolor{red}{R_{-1}}$=$\textcolor{blue}{B_0}$};
	\draw (6,0.5) node[anchor=north ] {$\textcolor{red}{R_{-3}}$=$\textcolor{red}{R_{0}}$=$\textcolor{blue}{B_{-1}}$};
	\draw (3.8,3) node[anchor=south west] {$\textcolor{blue}{B_{-2}}$};
	\draw (5.7,4) node[anchor=south] {$\textcolor{red}{R_{-2}}$};
	\draw (6.7,2.8) node[anchor= east] {$\textcolor{red}{e}$};
	\draw (6.53,2.8) node[anchor= west] {$\textcolor{blue}{o}$};
	\draw (5.6,3) node[anchor= north] {$\textcolor{red}{e}$};
	\draw (4.5,2.2) node[anchor= west] {$\textcolor{blue}{e}$};
	\draw (6.3,4.3) node[anchor= north] {$\textcolor{red}{?}$};
	\draw (7.1,3.6) node[anchor= north] {$T_j$};
\end{tikzpicture}
\caption{Opposite orientation through $T_j$: case that $R_{-2}\neq B_{-2}$.}
\label{RB}
\end{figure}
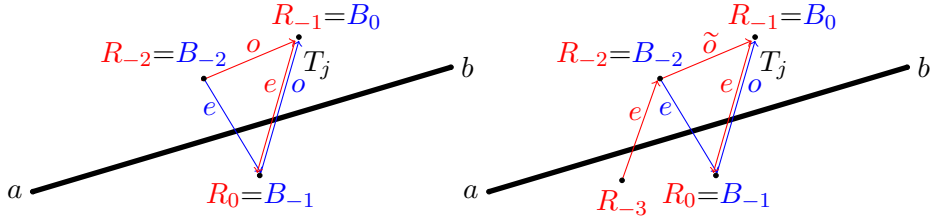
\begin{figure}
\centering
\begin{tikzpicture}[scale=0.5]
	\draw [line width=2pt](0,0)--(11,3.3);
	\draw (0,0) node[anchor=east]{$a$};
	\draw (11,3.3) node[anchor=west]{$b$};
	\draw[blue][->] (6,0.5)--(7,4);
	\draw[blue][->] (4.5,3)--(6,0.5);
	\draw[red][->] (6.9,4)--(5.9,0.5);
	\draw[red][->] (4.43,2.95)--(6.9,4);
	%\draw[red][->] (3.5,0.3)--(4.43,2.95);
	\filldraw (0,0) circle (0.06);
	\filldraw (11,3.3) circle (0.06);
	\filldraw (7,4.1) circle (0.06);
	\filldraw (5.97,0.43) circle (0.06);
	%\filldraw (3.5,0.3) circle (0.06);
	\filldraw (4.5,3) circle (0.06);
	\draw (6.,4) node[anchor=south west] {$\textcolor{red}{R_{-1}}$=$\textcolor{blue}{B_0}$};
	\draw (6,0.5) node[anchor=north ] {$\textcolor{red}{R_{0}}$=$\textcolor{blue}{B_{-1}}$};
	\draw (1.5,3) node[anchor=south west] {$\textcolor{red}{R_{-2}}$=$\textcolor{blue}{B_{-2}}$};
	%\draw (3.5,0.3) node[anchor=north ] {$\textcolor{red}{R_{-3}}$};
	\draw (6.7,2.8) node[anchor= east] {$\textcolor{red}{e}$};
	\draw (6.53,2.8) node[anchor= west] {$\textcolor{blue}{o}$};
	%\draw (3.82,2) node[anchor= north] {$\textcolor{red}{e}$};
	\draw (4.2,2.2) node[anchor= west] {$\textcolor{blue}{e}$};
	\draw (5.8,4.3) node[anchor= north] {$\textcolor{red}{o}$};
	\draw (7.5,4) node[anchor= north] {$T_j$};
	
	\draw [line width=2pt](12,0)--(23,3.3);
	\draw (12,0) node[anchor=east]{$a$};
	\draw (23,3.3) node[anchor=west]{$b$};
	\draw[blue][->] (18,0.5)--(19,4);
	\draw[blue][->] (16.5,3)--(18,0.5);
	\draw[red][->] (18.9,4)--(17.9,0.5);
	\draw[red][->] (16.43,2.95)--(18.9,4);
	\draw[red][->] (15.5,0.3)--(16.43,2.95);
	\filldraw (12,0) circle (0.06);
	\filldraw (23,3.3) circle (0.06);
	\filldraw (19,4.1) circle (0.06);
	\filldraw (17.97,0.43) circle (0.06);
	\filldraw (15.5,0.3) circle (0.06);
	\filldraw (16.5,3) circle (0.06);
	\draw (18,4) node[anchor=south west] {$\textcolor{red}{R_{-1}}$=$\textcolor{blue}{B_0}$};
	\draw (18,0.5) node[anchor=north ] {$\textcolor{red}{R_{0}}$=$\textcolor{blue}{B_{-1}}$};
	\draw (13.5,3) node[anchor=south west] {$\textcolor{red}{R_{-2}}$=$\textcolor{blue}{B_{-2}}$};
	\draw (15.5,0.3) node[anchor=north ] {$\textcolor{red}{R_{-3}}$};
	\draw (18.7,2.8) node[anchor= east] {$\textcolor{red}{e}$};
	\draw (18.53,2.8) node[anchor= west] {$\textcolor{blue}{o}$};
	\draw (15.85,2.5) node[anchor= north] {$\textcolor{red}{e}$};
	\draw (16.2,2.2) node[anchor= west] {$\textcolor{blue}{e}$};
	\draw (17.8,4.5) node[anchor= north] {$\textcolor{red}{\widetilde{o}}$};
	\draw (19.5,4) node[anchor= north] {$T_j$};
\end{tikzpicture}
\caption{Opposite orientation through $T_j$: local cases that $R_{-2}= B_{-2}$.}
\label{R-3}
\end{figure}
\item If $R_{-2}= B_{-2}$, then $R_{-2}R_{-1}$ must be a boundary. Otherwise, by $(\text{T3})-(\text{T6})$, the term $x_{R_{-2}R_{-1}}$ in $x(\mcP_1)$ does not appear in $x(\mcP_2)$, see the left part of \Cref{R-3} and we exclude it. Furthermore, we claim that $R_{-3}=R_{0}=B_{-1}$. Otherwise, by $(\text{T3})-(\text{T6})$, we get $R_{-3}R_{-2}$ crosses $ab$ and $R_{-3}$ lies on the left hand side of $R_0$, see the right part of \Cref{R-3}. Then, the term $x^{-1}_{B_{-2}B_{-1}}$ in $x(\mcP_2)$ does not appear in $x(\mcP_1)$ and we exclude it. To ensure the terms except $x_j$ in $x(\mcP_1)$ and $x(\mcP_2)$ are the same, by $(\text{T3})-(\text{T6})$, both $B_{-3}B_{-2}$ and $R_{-3}R_{-3}$ need to be boundaries. In addition, we have \begin{align}R_{-5}=B_{-3}, \ B_{-4}=R_{-4},\end{align} which provide the same term $x^{-1}_{R_{-5}R_{-4}}=x^{-1}_{B_{-4}B_{-3}}$ in $x(\mcP_1)$ and $x(\mcP_2)$. By analogy, more possible quadrangles consisted of two diagonals and two boundaries such as $B_{-4}B_{-3}B_{-2}B_{-1}$ are generated and the two $T$-paths will end up with the same endpoint $a$. Without loss of generality, on the other side of $T_j$, there is a similar phenomenon and the two $T$-paths will end up with the same endpoint $b$, see \Cref{R=B}. Then, all the terms except $x_j$ in $x(\mcP_1)$ and $x(\mcP_2)$ are same. Furthermore, by \Cref{coeff 012} and \Cref{prop-rmk}, we obtain that there is only one such $T$-path $\mcP_1$ and $\mcP_2$ if existing, which implies that \begin{align}a_{i_1,\dots,-1,\dots,i_m}=a_{i_1,\dots,1,\dots,i_m}=1.\end{align} 
\begin{figure}
\centering
\begin{tikzpicture}[scale=0.8]
	\draw [line width=2pt](0,0)--(11,3.3);
	\draw (0,0) node[anchor=east]{$a$};
	\draw (11,3.3) node[anchor=west]{$b$};
	\draw[blue][->] (6,0.5)--(7,4);
	\draw[blue][->] (4.5,3)--(5.8,0.5);
	\draw[blue][->] (3,2)--(4.42,3);
	\draw[red][->] (2.96,1.95)--(3.46,0.55);
	\draw[red][->] (6.92,4)--(5.92,0.5);
	\draw[red][->] (4.5,3)--(6.9,4);
	\draw[red][->] (5.9,0.43)--(4.57,3);
	\draw[red][->] (3.5,0.5)--(5.8,0.45);
	\draw[blue][->] (3.52,0.5)--(3.02,1.93);
	\draw[blue][->] (7.03,4)--(8,1.2);
	\draw[red][->] (8.09,1.2)--(7.12,4);
	\draw[red][->] (5.98,0.45)--(7.94,1.05);
	\draw[red][->] (7.1,4.1)--(8.6,3.98);
	\draw[blue][->] (7.94,1.05)--(8.91,1.75);
	\draw[blue][->] (8.97,1.75)--(8.63,3.9);
	\draw[red][->] (8.73,3.9)--(9.04,1.88);
	\draw [ dotted, line width=1pt] (3,2)--(5.88,0.43);
	\draw [blue][->] (2.2,0.23)--(3.4,0.55);
	\draw [red][->] (1.6,1.3)--(2.92,2);
	\draw [red][->] (9.04,1.88)--(9.8,2.5);
	\draw [blue][->] (8.73,3.95)--(9.7,3.8);
	\draw [ dotted, line width=1pt] (7,4.1)--(9,1.8);
	\draw [blue][->] (0,0)--(1.2,0.1);
	\draw (1.75,0.32) node[anchor=north] {$\dots$};
	\draw [red][->] (0,0)--(1,1);
	\draw (1.35,1.35) node[anchor=north] {$\dots$};
	\draw (2.4,1.2) node[anchor=north] {$\dots$};
	\draw [blue][->] (10.4,3.8)--(11,3.4);
	\draw (10.1,4) node[anchor=north] {$\dots$};
	\draw [red][->] (10.5,2.8)--(11,3.2);
	\draw (10.2,2.84) node[anchor=north] {$\dots$};
	\draw (9.7,3.4) node[anchor=north] {$\dots$};
	\filldraw (1.27,0.1) circle (0.06);
	\filldraw (1.05,1.05) circle (0.06);
	\filldraw (2.13,0.2) circle (0.06);
	\filldraw (1.6,1.3) circle (0.06);
	\filldraw (10.5,2.8) circle (0.06);
	\filldraw (10.4,3.8) circle (0.06);
	\filldraw (9.84,2.55) circle (0.06);
	\filldraw (9.74,3.8) circle (0.06);
	\filldraw (0,0) circle (0.06);
	\filldraw (11,3.3) circle (0.06);
	\filldraw (7,4.1) circle (0.06);
	\filldraw (5.88,0.43) circle (0.06);
	\filldraw (4.5,3) circle (0.06);
	\filldraw (3,2) circle (0.06);
	\filldraw (3.5,0.5) circle (0.06);
	\filldraw (8.04,1.1) circle (0.06);
	\filldraw (8.7,4) circle (0.06);
	\filldraw (9,1.8) circle (0.06);
	\draw (7,4) node[anchor=south] {$\textcolor{red}{R_{-1}}$=$\textcolor{red}{R_{2}}$=$\textcolor{blue}{B_0}$};
	\draw (6,0.5) node[anchor=north ] {$\textcolor{red}{R_{-3}}$=$\textcolor{red}{R_{0}}$=$\textcolor{blue}{B_{-1}}$};
	\draw (4.5,3.3) node[anchor=south ] {$\textcolor{red}{R_{-2}}$=$\textcolor{blue}{B_{-2}}$};
	\draw (3.4,2) node[anchor=south east ] {$\textcolor{red}{R_{-5}}$=$\textcolor{blue}{B_{-3}}$};
	\draw (3.5,0.3) node[anchor= north ] {$\textcolor{red}{R_{-4}}$=$\textcolor{blue}{B_{-4}}$};
	\draw (8.2,0.3) node[anchor=south ] {$\textcolor{red}{R_{1}}$=$\textcolor{blue}{B_{1}}$};
	\draw (9.4,1) node[anchor=south ] {$\textcolor{red}{R_{4}}$=$\textcolor{blue}{B_{2}}$};
	\draw (9.3,4) node[anchor=south ] {$\textcolor{red}{R_{3}}$=$\textcolor{blue}{B_{3}}$};
	\draw (6.7,2.8) node[anchor= east] {$\textcolor{red}{e}$};
	\draw (6.1,3.8) node[anchor= east] {$\textcolor{red}{\tilde{o}}$};
	\draw (6.53,2.8) node[anchor= west] {$\textcolor{blue}{o}$};
    \draw (3.4,2.75) node[anchor= west] {$\textcolor{blue}{\tilde{o}}$};
	\draw (4.5,2.1) node[anchor= west] {$\textcolor{blue}{e}$};
	\draw (4.85,2.3) node[anchor= west] {$\textcolor{red}{e}$};
	\draw (2.7,1.3) node[anchor= west] {$\textcolor{red}{e}$};
	\draw (3.1,1.4) node[anchor= west] {$\textcolor{blue}{e}$};
	\draw (7.25,2) node[anchor= west] {$\textcolor{blue}{e}$};
	\draw (7.68,2.1) node[anchor= west] {$\textcolor{red}{e}$};
	\draw (4.2,0.7) node[anchor= west] {$\textcolor{red}{\tilde{o}}$};
	\draw (6.7,1) node[anchor= west] {$\textcolor{red}{\tilde{o}}$};
	\draw (7.7,3.8) node[anchor= west] {$\textcolor{red}{\tilde{o}}$};
	\draw (8.2,1.7) node[anchor= west] {$\textcolor{blue}{\tilde{o}}$};
	\draw (8.3,3) node[anchor= west] {$\textcolor{blue}{e}$};
	\draw (8.7,3.1) node[anchor= west] {$\textcolor{red}{e}$};
	\draw (2,1.9) node[anchor= west] {$\textcolor{red}{\tilde{o}}$};
	\draw (0.2,0.8) node[anchor= west] {$\textcolor{red}{\tilde{o}}$};
	\draw (0.3,-0.2) node[anchor= west] {$\textcolor{blue}{\tilde{o}}$};
	\draw (2.5, 0.59) node[anchor= west] {$\textcolor{blue}{\tilde{o}}$};
	\draw (9.2,2) node[anchor= west] {$\textcolor{red}{\tilde{o}}$};
	\draw (10.5,2.8) node[anchor= west] {$\textcolor{red}{\tilde{o}}$};
	\draw (8.93,3.65) node[anchor= west] {$\textcolor{blue}{\tilde{o}}$};
	\draw (10.5,3.8) node[anchor= west] {$\textcolor{blue}{\tilde{o}}$};
	\draw (6.1,2.6) node[anchor= north] {$T_j$};
\end{tikzpicture}
\caption{Opposite orientation through $T_j$: overall case that $R_{-2}=B_{-2}$.}
\label{R=B}
\end{figure}
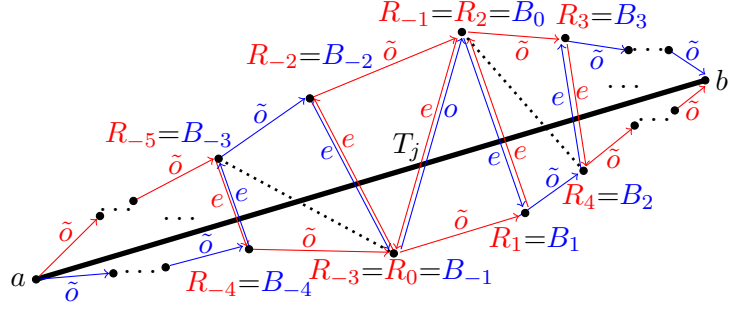
Now, we claim that $a_{i_1,\dots,0,\dots,i_m}=2$. Note that in \Cref{R=B}, the $T$-path $\mcP_1$ is indexed successively by vertices \begin{align}(a=R_s,\dots,R_{-4},R_{-3},R_{-2},R_{-1},R_{0},R_{1},R_{2},R_{3},\dots,R_t=b), \end{align}
		and $T$-path $\mcP_2$ is indexed successively by vertices \begin{align}(a=B_u,\dots,B_{-3},B_{-2},B_{-1},B_{0},B_{1},B_{2},\dots,B_v=b), \end{align}
		such that \begin{align}
			\left\{
		\begin{array}{ll} 
		R_{-3}=R_0=B_{-1},\\
		R_{-1}=R_2=B_{0},\\
			R_{2k+1}=B_{2k+1},\  k\geq 0, \\
			R_{2k+2}=B_{2k},\ \ \ \ k\geq 1,\\
			R_{2k}=B_{2k},\  \ \ \ \ \ \  k\leq -1,\\
			R_{2k-3}=B_{2k-1},\    k\leq -1.
		\end{array} \right.
		\end{align} Hence, there are two $T$-paths $\mcP_3$ and $\mcP_4$ which provide a same Laurent monomial \begin{align}x_1^{i_1}\dots x_{j-1}^{i_{j-1}}x_{j+1}^{i_{j+1}}\dots x_m^{i_m},\end{align} where $\mcP_3$ is indexed successively by vertices \begin{align}(a=B_u,\dots,B_{-3},B_{-2},B_{-1},R_{1},R_{2},R_{3}\dots R_t=b ),\end{align} and $\mcP_4$ is indexed successively by vertices \begin{align}(a=R_s,\dots,R_{-4},R_{-3},R_{-2},R_{-1},B_{1},B_{2},\dots,B_v=b).\end{align} Therefore, we conclude that if $a_{i_1,\dots,-1,\dots,i_m}a_{i_1,\dots,1,\dots,i_m}\neq 0$, then $a_{i_1,\dots,0,\dots,i_m}=2$.
	\end{enumerate}
 
In conclusion, it is direct that $2^2\geq 1\times 1$ and the inequality \eqref{inequality} holds, which implies that all the cluster variables of type $A_n$ are log-concave.
\end{proof}
\begin{remark}
	In \Cref{R=B}, the $T$-paths $\mcP_1$ and $\mcP_2$ together with their corresponding Laurent monomials $x(\mcP_1)$ and $x(\mcP_2)$ are independent of the choices of the diagonals in each quadrangle such as $B_{-4}B_{-3}B_{-2}B_{-1}$. Particularly, we use the black dotted lines to represent them. In addition, we can refer to \Cref{T-example} as an example of \Cref{main1}.
\end{remark}
\section{Log-concavity of cluster monomials of type $A_2$}
In this section, we aim to prove the log-concavity of all the cluster monomials of type $A_2$. Afterwards, we give a conjecture with respect to the general cases of type $A_n$ with $n\geq 3$.

Firstly, based on the Laurent phenomenon of cluster variables, we give a novel definition about cluster monomials.
\begin{definition}[\emph{Log-concavity of cluster monomials}]
	For a coefficient-free cluster algebra $\mcA$, a cluster monomial is called \emph{log-concave} if it is log-concave as a Laurent polynomial with respect to the initial cluster $\mfx$.
\end{definition}
Then, we give a preliminary lemma as follows.
\begin{lemma}\label{binomial2}
	For $0\leq k\leq n-1$, the inequality related with binomial coefficients holds: \begin{align}(C_n^k)^2\geq C_{n-1}^kC_{n+1}^k.\end{align}
\end{lemma}
\begin{proof}
	We only need to note that \begin{align}
		\frac{(C_n^k)^2}{C_{n-1}^{k}C_{n+1}^{k}}=\frac{(n-k+1)n}{(n-k)(n+1)}=\frac{n^2-kn+n}{n^2-kn+n-k}\geq 1. 
	\end{align}
	Hence, the inequality holds. 
\end{proof}
Now, we are ready to prove another main theorem.
\begin{theorem}\label{main theorem}
	The cluster monomials of type $A_2$ are log-concave.
\end{theorem}
\begin{proof} According to \Cref{A2 type}, there are five classes of cluster monomials and we prove that they are log-concave one by one.
\begin{enumerate}[leftmargin=2em]
	\item It is direct that the cluster monomials $x_1^{m_1}x_2^{m_2}$ are log-concave.  
	\item Since $(x_2+1)^{m_1}$ is log-concave, by \Cref{same up}, we get $(\frac{x_2+1}{x_1})^{m_1}x_2^{m_2}$ is log-concave.
	\item In order to prove $(\frac{x_2+1}{x_1})^{m_1}(\frac{x_1+x_2+1}{x_1x_2})^{m_2}$ is log-concave, we only need to prove that $T_{m_1,m_2}=(x_2+1)^{m_1}(x_1+x_2+1)^{m_2}$ is log-concave. Note that
	\begin{align}
	T_{m_1,m_2}
	=&{}(x_2+1)^{m_1}\sum\limits_{i=0}^{m_2}C_{m_2}^ix_1^{m_2-i}\left(x_2+1\right)^i\notag \\ =&{}\sum\limits_{i=0}^{m_2}C_{m_2}^ix_1^{m_2-i}(x_2+1)^{m_1+i}\notag \\
	=&{} \sum\limits_{i=0}^{m_2}C_{m_2}^ix_1^{m_2-i}\left(\sum\limits_{j=0}^{m_1+i}C_{m_1+i}^jx_2^j\right) \notag \\
	=&{}\sum\limits_{i=0}^{m_2}\sum\limits_{j=0}^{m_1+i}C_{m_2}^iC_{m_1+i}^jx_1^{m_2-i}x_2^j\notag.
\end{align}
Let $k=m_2-i$ and $l=j$. Then we have 
\begin{align}\label{eq1}
	\sum\limits_{i=0}^{m_2}\sum\limits_{j=0}^{m_1+i}C_{m_2}^iC_{m_1+i}^jx_1^{m_2-i}x_2^j=\sum\limits_{k=0}^{m_2}\sum\limits_{l=0}^{m_1+m_2-k}C_{m_2}^{k}C_{m_1+m_2-k}^lx_1^{k}x_2^{l}.
\end{align} Note that for each monomial $x_1^kx_2^l$ in \eqref{eq1}, the corresponding coefficient consists of a single term and we denote it by \begin{align}a_{k,l}=C_{m_2}^{k}C_{m_1+m_2-k}^l.\end{align} According to \Cref{log-concave examples} and \Cref{binomial2}, we conclude that
 \begin{align}
 	a_{k,l}^2=&{}(C_{m_2}^{k}C_{m_1+m_2-k}^l)^2 \notag \\ 	\geq&{}C_{m_2}^{k-1}C_{m_2}^{k+1}C_{m_1+m_2-(k-1)}^lC_{m_1+m_2-(k+1)}^l\notag \\
 	=&{} a_{k-1,l}a_{k+1,l}\notag
 \end{align}
 Similarly, we have $a_{k,l}^2\geq a_{k,l-1}a_{k,l+1}$.
\item By symmetry, it is similar to the proof of (3).
\item Since $(x_1+1)^{m_2}$ is log-concave, by \Cref{same up}, we get $(\frac{x_1+1}{x_2})^{m_1}x_1^{m_2}$ is log-concave.
\end{enumerate} 
Therefore, all the cluster monomials of type $A_2$ are log-concave.
\end{proof}
\begin{remark}
	By use of the properties of binomial coefficients, we have proved the log-concavity of cluster monomials (theta functions) of type $A_2$. However, it is quite difficult to deal with the cases of type $A_n$ with $n\geq 3$, but we believe it is correct. Hence, we give a conjecture to finish this section. 
\end{remark}
\begin{conjecture}\label{An conj}
	The cluster monomials of type $A_n(n\geq 3)$ are log-concave.
\end{conjecture}

\section{Log-concavity of $F$-polynomials of type $A_n$}
In this section, we recall the notions of $C$-matrices, $G$-matrices and $F$-polynomials based on \cite{FZ07}. Furthermore, according to \cite{Gyo21}, we give the definitions of $f$-vectors and $F$-matrices.  We aim to prove the log-concavity of $F$-polynomials of type $A_n$ with $n\geq 2$.
\subsection{$F$-polynomials and $f$-vectors}\

Before giving the definition of $F$-polynomials, we recall two crucial notions of $C$-matrices and $G$-matrices as follows.
\begin{definition}[{\cite[equation (5.9), Proposition 6.6]{FZ07}}]
	Let $B_0$ be the initial exchange matrix at $t_0$. Then, the collections of integer square matrices $\big\{C_t^{B_0;t_0}\big\}_{t\in\mbT_n}$ and $\big\{G_t^{B_0;t_0}\big\}_{t\in \mbT_n}$ are recursively defined as follows:
	\begin{enumerate}
		\item ($C$-matrices) The initial condition is
\begin{align}
C^{B_0; t_0}_{t_0}=I_n,
\end{align}
and for any edge $t \stackrel{k}{\longleftrightarrow} t^{\prime}$ in $\mbT_n$, the recursion formula is 
\begin{align}
	C^{B_0; t_0}_{t'}=C^{B_0; t_0}_t\big(J_k+[ B_t]^{k \bullet}_+\big)+\big[- C^{B_0; t_0}_t\big]^{\bullet k}_+ B_t.
\end{align}
\item ($G$-matrices) The initial condition is
\begin{align}
G^{B_0; t_0}_{t_0}=I_n,
\end{align}
and for any edge $t \stackrel{k}{\longleftrightarrow} t^{\prime}$ in $\mbT_n$, the recursion formula is 
\begin{align}
	G^{B_0; t_0}_{t'}=G^{B_0; t_0}_t\big(J_k+[ B_t]^{\bullet k}_+\big)-B_0\big[ C^{B_0; t_0}_t\big]^{\bullet k}_+.
\end{align}
	\end{enumerate}
\end{definition}
\begin{remark}\label{CG}
	By {\cite[Equation (6.14)]{FZ07}}, there is a duality between $C$-matrices and $G$-matrices:
	\begin{align}
	B_0C^{B_0; t_0}_t=G^{B_0; t_0}_tB_t.
	\end{align}
\end{remark}
\begin{definition}[\emph{Principal coefficients}]
	A cluster algebra $\mcA$ is said to have \emph{principal coefficients} at vertex $t_0$ if $\mbP=\text{Trop}(y_1,\dots,y_n)$ and $\mfy_{t_0}=(y_1,\dots,y_n)$.
\end{definition}
When a cluster algebra $\mcA$ has principal coefficients, we can define the \emph{$F$-polynomial} $F^{B_0;t_0}_{i;t}(\mfy)$ as follows:
\begin{align}
F^{B_0;t_0}_{i;t}(\mfy)=x_{i;t}(x_1,\dots, x_n)|_{x_1=\cdots=x_n=1}=N_{i;t}(1,\dots, 1).
\end{align}
Denote the maximal degree of $y_j$ in $F_{i;t}^{B_0;t_0}(\mfy)$ by $f_{ji;t}$. Then, we can define the \emph{$f$-vector} $\mff_{i;t}^{B_0;t_0}$ and \emph{$F$-matrices} $F_t^{B_0;t_0}$ as follows:
\begin{align}
\mff_{i;t}^{B_0;t_0}=\mff_{i;t}=\begin{pmatrix}f_{1i;t}\\ \vdots \\ f_{ni;t} \end{pmatrix},\ F_t^{B_0;t_0}=F_t=(\mff_{1;t},\dots,\mff_{n;t}).
\end{align}
\begin{definition}[{\cite[Corollary 6.3]{FZ07}}]\label{separation}
	Let $\mcA$ be a coefficient-free cluster algebra with the initial exchange matrix $B_0$. Then, for any $t \in \mbT_n$ and $i \in \{1, \dots, n\}$, we have
	\begin{align}
		x_{i;t} = \big( \prod_{j=1}^n x_j^{g_{ji;t}} \big)F_{i;t}^{B_0;t_0}(\hat y_1, \dots, \hat y_n) %\frac{F_{i;t}^{B_0;t_0}|_\mcF(\hat y_1, \dots, \hat y_n)}{F_{i;t}^{B_0;t_0}|_\mbP (y_1, \dots, y_n)},
		\label{separation formula}
	\end{align}
where $\hat y_k = \mathop{\prod}\limits_{j=1}^n x_j^{b_{jk}}$ and $g_{ji;t}$ is the $(j,i)$-element of $G_t^{B_0;t_0}$. The formula \eqref{separation formula} is usually called the  \emph{coefficient-free separation formula} of $x_{i;t}$.
\end{definition}
\begin{remark}
	By \Cref{main1} and the separation formula \eqref{separation formula}, we can directly deduce that the Laurent polynomial $F_{i;t}^{B_0;t_0}(\hat y_1, \dots, \hat y_n)$ is log-concave with respect to $x_i$. However, it can't directly imply that each $F$-polynomial $F_{i;t}^{B_0;t_0}( y_1, \dots,  y_n)$ is log-concave with respect to $y_i$.  In the next subsection, we aim to prove this fact. 
\end{remark}
Firstly, there is an important relation between $f$-vectors and $d$-vectors of cluster algebras of finite type (such as type $A_n$) as follows.
\begin{theorem}{\cite[Theorem 1.8]{Gyo21}}\label{gyo21}
	In a cluster algebra $\mcA$ of finite type, for any $i\in \{1,\dots,n\}$ and $t\in \mbT_n$, we have 
	\begin{align}
		\mff_{i;t}=[\mfd_{i;t}]_+.
	\end{align}
\end{theorem}

\subsection{Geometric realization of $d$-vectors of type $A_n$}\

In this subsection, we use the geometric realization of $d$-vectors of type $A_n$ and prove the log-concavity of $F$-polynomials of type $A_n$. Following the geometric realization of cluster algebras of type $A_n$ in \Cref{3.1}, we let $T$ be the initial triangulation of $(n+3)$-gon, which corresponds to the initial cluster $\mfx=\mfx_{T}$.
\begin{proposition}{\cite[Theorem 8.6]{FST08}}
\label{geo d-vector}
	Let $\mcA$ be a cluster algebra of type $A_n$ and $x_{\gamma}$ be the cluster variable corresponding to the diagonal $\gamma$, where $\gamma \notin T$. Then its corresponding $d$-vector is 
	\begin{align}
		\mfd_\gamma= \begin{pmatrix}d_{1}\\ \vdots \\ d_{n} \end{pmatrix},
	\end{align}
	where $d_{j}$ is the number of intersections between $\gamma$ and $T_j$.
\end{proposition}
Based on the preparations above, we are ready to prove the theorem as follows.
\begin{theorem}\label{log-concavity of F poly}
	All the $F$-polynomials of type $A_n$ are log-concave.
\end{theorem}
\begin{proof}\
There are two possible cases of $F$-polynomials to be discussed.
	\begin{enumerate}[leftmargin=2em]
		\item If the cluster variable $x_{i;t}$ belongs to the initial cluster $\mfx$, then $F_{i;t}(y_1,\dots,y_n)=1$ and it is log-concave immediately.
		\item If the cluster variable $x_{i;t}$ does not belong to the initial cluster $\mfx$, then by the geometric realization and \Cref{geo d-vector}, we have \begin{align}0\leq d_{ji;t}\leq 1\end{align} for any $j\in \{1,\dots,n\}$. Hence, according to \Cref{gyo21}, we have \begin{align}0\leq f_{ji;t}\leq 1\end{align} for any $j\in \{1,\dots,n\}$. As a consequence, by the definition of $f$-vectors, we conclude that $F_{i;t}(y_1,\dots,y_n)$ is log-concave.
	\end{enumerate}
\end{proof}
\begin{example} The $C$-matrices, $D$-matrices (consisting of $d$-vectors), $G$-matrices and $F$-polynomials of type $A_2$ are as follows, see \Cref{A2type}. It is direct that the $F$-polynomials of type $A_2$ are log-concave.
\begin{table}[ht]

\centering
\scalebox{0.67}{
\begin{tabular}{|c|c|c|c|c|c|}
\hline
&&&&&\\[-3mm]
Type $A_2$& $t_0$ & $t_1$ & $t_2$  & $t_3$  & $t_4$ \\ [1mm]
\hline
&&&&&\\[-2mm]
$B$-matrix& $\begin{pmatrix}
 0 & 1 \\
 -1 & 0
\end{pmatrix}$ & $\begin{pmatrix}
 0 & -1 \\
 1 & 0
\end{pmatrix}$
& $\begin{pmatrix}
 0 & 1 \\
 -1 & 0
\end{pmatrix}$
& $\begin{pmatrix}
 0 & -1 \\
 1 & 0
\end{pmatrix}$ & $\begin{pmatrix}
 0 & 1 \\
 -1 & 0
\end{pmatrix}$
 \\[4mm] 
\hline
&&&&&\\ [-2mm]
Cluster& $\big(x_1,x_2\big)$ & $\big(\dfrac{x_2+y_1}{x_1},x_2\big)$
& $\big(\dfrac{x_2+y_1}{x_1},\dfrac{y_1y_2x_1+x_2+y_1}{x_1x_2}\big)$
& $\big(\dfrac{y_2x_1+1}{x_2},\dfrac{y_1y_2x_1+x_2+y_1}{x_1x_2}\big)$ & $\big(\dfrac{y_2x_1+1}{x_2},x_1\big)$ \\[4mm]
\hline
&&&&&\\[-2mm]
$C$-matrix& $\begin{pmatrix}
 1 & 0 \\
 0 & 1
\end{pmatrix}$ & $\begin{pmatrix}
 -1 & 1 \\
 0 & 1
\end{pmatrix}$
& $\begin{pmatrix}
 0 & -1 \\
 1 & -1
\end{pmatrix}$
& $\begin{pmatrix}
 0 & -1 \\
 -1 & 0
\end{pmatrix}$ & $\begin{pmatrix}
 0 & 1 \\
 -1 & 0
\end{pmatrix}$ \\[4mm]
\hline
&&&&&\\[-2mm]
$D$-matrix& $\begin{pmatrix}
 -1 & 0 \\
 0 & -1
\end{pmatrix}$ & $\begin{pmatrix}
 1 & 0 \\
 0 & -1
\end{pmatrix}$
& $\begin{pmatrix}
 1 & 1 \\
 0 & 1
\end{pmatrix}$
& $\begin{pmatrix}
 0 & 1 \\
 1 & 1
\end{pmatrix}$ & $\begin{pmatrix}
 0 & -1 \\
 1 & 0
\end{pmatrix}$ \\[4mm]
\hline
&&&&&\\[-2mm]
$G$-matrix& $\begin{pmatrix}
 1 & 0 \\
 0& 1
\end{pmatrix}$ & $\begin{pmatrix}
 -1 & 0 \\
 1& 1
\end{pmatrix}$
& $\begin{pmatrix}
 1 & -1 \\
 1 & 0
\end{pmatrix}$
& $\begin{pmatrix}
 0 & -1 \\
 -1 & 0
\end{pmatrix}$ & $\begin{pmatrix}
 0 & 1 \\
 -1 & 0
\end{pmatrix}$\\[4mm]
\hline
&&&&&\\ [-1mm]
$F$-polynomial& $1$, $1$& $y_1+1$, $1$
& $y_1+1$,  $y_1y_2+y_1+1$
& $y_2+1$, $y_1y_2+y_1+1$ & $y_2+1$, $1$\\[3.5mm]
\hline
&&&&&\\[-2mm]
$F$-matrix& $\begin{pmatrix}
 0 & 0 \\
 0 & 0
\end{pmatrix}$ & $\begin{pmatrix}
 1 & 0 \\
 0 & 0
\end{pmatrix}$
& $\begin{pmatrix}
 1 & 1 \\
 0 & 1
\end{pmatrix}$ & $\begin{pmatrix}
 0 & 1 \\
 1 & 1
\end{pmatrix}$ & $\begin{pmatrix}
 0 & 0 \\
 1 & 0
\end{pmatrix}$\\[4mm]
\hline
\end{tabular}}
\hspace{3cm}
\caption{Cluster information of type $A_2$ with principal coefficients.} \label{A2type}
\end{table}
	
\end{example}
%\section{Further discussion about block multi-convexity}
%\begin{definition}
%	A function $f(x_1,\dots,x_m): \mbR^{m}\rightarrow \mbR$ is called \emph{block multi-convex} if it is convex with respect to each individual variable. More precisely, for each $i\in \{1,\dots,m\}$, fix all the variables except $x_i$ and let $f_{x_1,\dots,\widehat{x_i},\dots,x_m}(x_i)=f(x_1,\dots,x_m)$, then $f_{x_1,\dots,\widehat{x_i},\dots,x_m}(x_i)$ is convex over $\mbR$.
%\end{definition}

\end{document}